\documentclass[11pt]{amsart}
\usepackage[T1]{fontenc}
\usepackage[ansinew]{inputenc}
\usepackage{amssymb,amsopn,amsmath,amsthm,authblk,hyperref, cleveref,lipsum,scrextend, mathrsfs,array,graphicx,xcolor,bbm,enumerate,wasysym,yfonts, ulem,dsfont}
\usepackage{stmaryrd}
\usepackage{subcaption}
\usepackage{verbatim}
\usepackage[all]{xy}
\usepackage{tikz}
\usepackage[displaymath, mathlines]{lineno}
\usepackage{fullpage}
\usepackage[foot]{amsaddr}
\usepackage{xifthen}
\usepackage{xstring}
\usepackage{xparse}
\usepackage{float}
\usepackage{ulem}

\newtheorem{thm}{Theorem}[section]
\newtheorem{lemma}[thm]{Lemma}
\newtheorem{rem}[thm]{Remark}
\newtheorem{defn}[thm]{Definition}
\newtheorem{claim}[thm]{Claim}
\newtheorem{prop}[thm]{Proposition}
\newtheorem{cor}[thm]{Corollary}

\newcommand{\eqd}{\stackrel{(d)}{=}}
\newcommand{\sh}{\mathsf{ S}}
\newcommand{\Ball}[2]{\mathcal{B}_{#2}(#1)}
\newcommand{\Crt}{\mathcal{T}}

\newcommand{\TT}{\mathcal{T}}
\newcommand{\ttm}{\mathfrak{t}^{\mathsf{M}}}

\newcommand{\Q}{\mathbb Q}

\newcommand{\QQQ}{\mathfrak Q}
\newcommand{\HHH}{\mathbf{H}}
\newcommand{\Z}{\mathbb Z}
\newcommand{\N}{\mathbb N}
\newcommand{\R}{\mathbb R}
\newcommand{\E}{\mathbb E}
\newcommand{\UIHPQ}{\text{UIHPQ}}
\newcommand{\e}{\mathbbm e}

\renewcommand{\S}{\mathbf S}
\renewcommand{\P}{\mathbb P}
\newcommand{\Tr}[2]{\mathcal{R}_{#2}(#1)}
\renewcommand{\1}{\mathbf 1}
\newcommand{\mm}{\mathfrak{m}}

\newcommand{\pp}{\mathfrak{p}}
\newcommand{\qq}{\mathfrak{q}}
\renewcommand{\tt}{\mathfrak{t}}
\newcommand{\rqq}{\mathbf{q}}

\newcommand{\rqs}{\mathbf{q}_s}

\newcommand{\rtt}{\mathbf{t}}
\newcommand{\Rtt}{\mathbf{T}}
\newcommand{\rbb}{\mathbf{b}}
\newcommand{\Rbb}{\mathbf{B}}
\newcommand{\rtq}{\mathbf{t}^{\mathsf{M}}}

\newcommand{\Rmm}{\mathbf{M}}
\newcommand{\Smm}{\mathcal{S}}
\newcommand{\Stm}{\mathbf{T}^{\mathsf{S}}}

\newcommand{\dd}{\mathsf{d}}

\newcommand{\rtm}{\mathbf{t}^{\mathsf{M}}}

\newcommand{\gtt}{\gamma^{\rtt}}
\newcommand{\Gtt}{\gamma^{\Rtt}}
\newcommand{\Gs}{\gamma^{\mathbf{S}}}
\newcommand{\gm}{\gamma^{\rqq}}
\newcommand{\Gm}{\gamma^{\Rmm}}
\newcommand{\gb}{\gamma^{\rbb}}
\newcommand{\Gb}{\gamma^{\Rbb}}
\newcommand{\Gh}{\gamma^{\HHH}}
\newcommand{\Gr}{\gamma^{\Rmm}}

\newcommand{\qt}[2]{\mathbf{q}_{#1}^{\mathsf{T},#2}}

\newcommand{\bb}{\mathfrak{b}}
\def\cro#1{\llbracket#1\rrbracket}

\makeatletter
\g@addto@macro{\endabstract}{\@setabstract}
\newcommand{\authorfootnotes}{\renewcommand\thefootnote{\@fnsymbol\c@footnote}}
\makeatother

\let \epsilon \varepsilon

\numberwithin{equation}{section}

\crefformat{footnote}{#2\footnotemark[#1]#3}

\title{On the triviality of the Shocked-map}
\author{Luis Fredes$^*$}
\address{$^*$Univ. Bordeaux, CNRS, Bordeaux INP, IMB, UMR 5251, F-33400 Talence, France.}
\author{Avelio Sepúlveda$^\dagger$}
\address{$^\dagger$ Universidad de Chile,  Centro de Modelamiento Matemático (AFB170001), UMI-CNRS 2807, Beauchef 851, Santiago, Chile.}
\begin{document}
	\maketitle
	\begin{abstract}
	The (non-spanning) tree-decorated quadrangulation is a random pair formed by a quadrangulation and a subtree chosen uniformly over the set of pairs with prescribed size. In this paper we study the tree-decorated quadrangulation in the critical regime: when the number of faces of the map, $f$, is proportional to the square of the size of the tree. We show that with high probability in this regime,  the diameter of the tree is between $o(f^{1/4})$ and  $f^{1/4}/\log^\alpha(f)$, for $\alpha >1$. Thus after scaling the distances by $f^{-1/4}$, the critical tree-decorated quadrangulation converges to a Brownian disk where the boundary has been identified to a point. These results imply the triviality of the shocked map: the metric space generated by gluing a Brownian disk with a continuous random tree.
	\end{abstract}
	\section{Introduction} 	
\begin{figure}[h!]
	\includegraphics[width=0.8\textwidth]{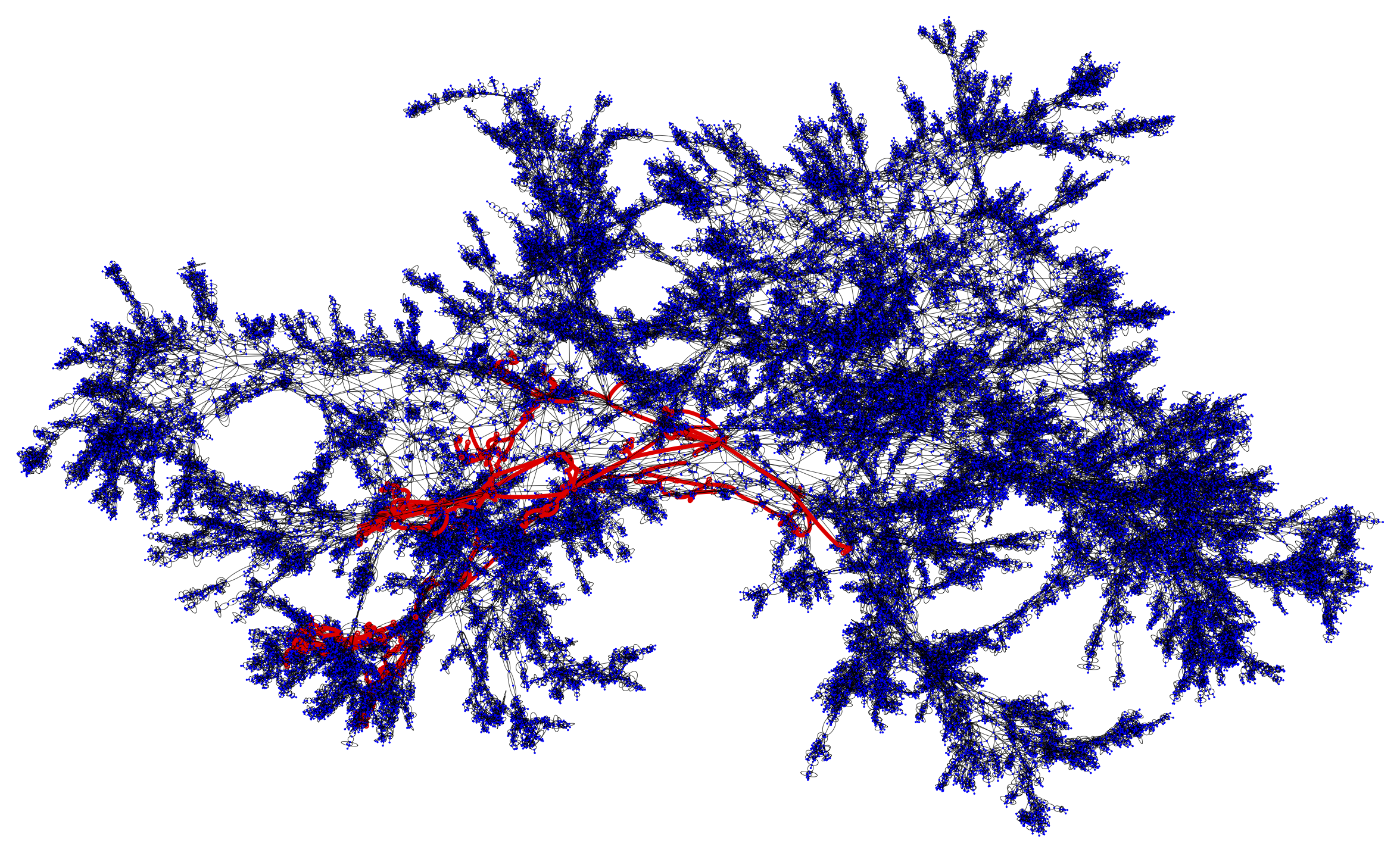}
	\caption{A simulation, based on the bijection introduced in \cite{FS19}, of a uniformly chosen tree decorated quadrangulation.}
\end{figure} 

The (non-spanning) tree-decorated map was introduced in \cite{FS19}. It consists in a uniformly chosen couple $(\rqq,\rtq)$, where $\rqq$ is a planar quadrangulation with $f$ faces and $\rtq$ is a subtree of $\rqq$ with $k$ edges containing the root of $\rqq$. The main reason for its introduction was to propose a model that interpolates between the uniformly chosen planar quadrangulation and the spanning-tree decorated quadrangulation. This is interesting as these two models belong to different universality classes, as shown in \cite{She,GHS2,DGExpB,GHS}.

In this work, we discuss the scaling limit of the tree-decorated map when $k \propto \sqrt f$. Let us first justify why this regime should be critical.

In \cite{FS19}, we introduced a bijection between the set of tree-decorated quadrangulations with $f$ faces decorated by a tree with $k$ edges, and the Cartesian product of two sets: the set of planar trees of size $k$ and the set of planar quadrangulation with a simple boundary with $f$ internal faces and boundary of size $2k$. The bijection is simple and can be informally understood as a ``gluing'' of the boundary using the equivalence relationship defined by the tree. This bijection is close to the one found in \cite{ BJ11, GM2,CC19} to study the planar maps decorated by a self-avoiding walk.

The bijection of \cite{FS19} gives interesting information about the possible scaling limits of tree-decorated maps. For example, as the tree chosen is always uniform, it is clear that the tree $\rtm$ (properly normalised) converges to a CRT as long as its size, $k$, goes to infinity. 

The planar map $\rqq$, however, is trickier. In this case, it is interesting to study the case of uniform quadrangulation with a simple boundary $\rqs$ that appears in the bijection. We start with a more elementary case : the uniform quadrangulation $\rqq_b$ with general boundary of size $2k$. It is known thanks to\cite{BETQ,GM2}, that this model undergoes a phase transition when $k\asymp \sqrt f$ 
\begin{align*}
f^{-\alpha}\rqq_b \asymp \begin{cases}
\text{Brownian map} & \text{if } k\ll\sqrt{f}\\
\text{Brownian disk} & \text{if } k \asymp \sqrt{f}\\
\text{Continuum random tree} & \text{if } k\gg \sqrt{f},
\end{cases}
\end{align*}
where we write $cM = (M,c\dd_M)$ for a rescaled discrete metric space $(M,\dd_M)$.
The simple boundary case is expected to be similar. The first two points are treated in \cite{GM,BCFS},  however the third point is still not proven.

The phase transition of quadrangulations with a simple boundary $\rqs$, allows us to conjecture that the bijection of \cite{FS19} also works in the scaling limit when $k \approx \sigma \sqrt f$. In fact, we call \textit{the shocked map} the result of the bijection of \cite{FS19} in the continuum (see Definition \ref{SMdef}). Thus, a shocked map is composed by a pair $(\Smm,\Stm)$ where $\Smm$ is a continuous compact metric space and $\Stm$ is a subset of $\Smm$.

\subsection{Main results}

The main result of this paper is that the shocked map is trivial. More precisely after doing the glueing, the metric space $\Smm$ does not reduce to one-point, but the set $\Stm$ does; i.e. after the glueing the tree contracts into a point. This result is summarized in the following theorem.
\begin{thm}\label{t.triviality-SM}
	Let $(\Smm,\Stm)$ be a shocked map. One has that $\Smm$ is the Brownian disk where the boundary is identified to a point and $\Stm$ corresponds to that point.
\end{thm}
As a consequence of this theorem, we can describe the scaling limit of the critical tree-decorated map and obtain an upper bound for the asymptotic behavior of the diameter of the tree with respect to the metric of the map. This upper bound is tight up to a log correction as seen in the following theorem.
\begin{thm}\label{t.diam}
	Let $(\rqq_f, \rtq_\sigma)$ be a tree decorated quadrangulation where $\rqq_f$ has $f$ faces and the tree $\rtq_\sigma$ is of size $\sigma \sqrt f$, with $0<\sigma<\infty$.  Then, $(8f/9)^{-1/4}(\rqq_f, \rtq_\sigma)$ converges in law for the Gromov-Hausdorff-Uniform topology \footnote{The definition of this topology is given in \Cref{sec:GH}. For more information see for example \cite{BBI}} towards $(\Smm,\Stm)$.
	Furthermore, for any $\alpha>1$ and $\epsilon>0$, such that with high probability as $f\to \infty$.
	\begin{align}\label{e.bounds_finite_discrete}
		\frac{f^{1/4}}{(\log(f))^\alpha}\leq diam(\rtm_{\sigma})\leq \epsilon f^{1/4}.
	\end{align}
\end{thm}
The bounds in \eqref{e.bounds_finite_discrete}  are proven in Proposition \ref{p.TDM_small_diameter}. Theorem \ref{t.diam} follows directly from this. 
\begin{proof}[Proof of the convergence assuming \eqref{e.bounds_finite_discrete}] The sequence is pre-compact thanks to Remark \ref{compGHU}. Now, notice that when renormalising by $f^{1/4}$, the diameter of the tree $\rtm_{\sigma}$ converges to 0 meaning that it contracts to one point as $f\to\infty$ and notice that every path having empty intersection with the decoration keeps its length. This let us upper and lower bound the distances by the distances of the shock map as $f\to\infty$, giving the equality.
	
\end{proof}
To prove Theorem \ref{t.triviality-SM}, we first work in the continuous infinite volume version of the model. We start with $\HHH$ a Brownian half-plane and $\Stm$ a bi-infinite tree and we glue them according to boundary length. Then, we show, using Kingman's subadditive ergodic theorem, that there exists a deterministic constant so that a.s. $\lim n^{-1}d(0,\tau_n)= c$. Here, $\tau_n$ is the point in the right infinite branch at distance $n$ from $0$. Afterward, we show that $c$ is equal to $0$ by studying how this distance behaves on an event with positive probability. We conclude that the diameter is $0$ by using  the rerooting invariance of our objects and that $d(0,\tau_n)n^{-1}$ is equal in law to $d(0,\tau_1)$. 

As a consequence, we see that when one glues $\HHH$ with an infinite CRT the diameter of the image of CRT is 0, as the distance in this case are smaller than in the former one.

To obtain the results in the finite volume case, we use the ideas of \cite{BCFS} to see that when one only explores a part of the infinite CRT and of the boundary of the Brownian half-plane, their laws are absolutely continuous with respect to those of the CRT and the Brownian disk respectively. This shows that the diameter of the explored parts of the trees are also zero.

The upper bound of Theorem \ref{t.diam} is a direct consequence of Theorem \ref{t.triviality-SM}, as discrete distances are smaller than their continuum counterparts. However, the  lower bounds can only be obtained from the discrete. To do this, we work in the infinite volume limit of the tree-decorated map. In this case, one can see that the furthest point in the branch of the infinite tree that belongs to the ball of centre $0$ and radius $n$ is stochastically dominated by the sum of $n$ i.i.d. positive random variables with tail decreasing as $x^{-1}$. This implies that this furthest point has distance smaller than $n\log^\alpha(n)$. Again, the finite volume case follows by absolute continuity.

The paper is organised as follows, we start with the preliminaries, where we introduce all the necessary objects and results for what follows. Then, in Section \ref{s.ic} we work on the case where the map is continuous and the volume is infinite to see that the diameter of the tree is $0$. In Section \ref{s.id}, we work in the case where the map is discrete and the volume is infinite, to obtain lower bounds on the distances in the infinite spine. In Section \ref{s.f_volume_d_m} and \ref{s.fc}, we work in the case where the volume of the map is finite and the map is discrete and continuous respectively.

\subsection*{Acknowledgements}
We thank Armand Riera for fruitful discussions. The research of A.S is supported by Grant ANID AFB170001, FONDECYT iniciación de investigación N° 11200085 and ERC 101043450 Vortex. Part of this work was done while L.F was working at University Paris-Saclay, he acknowledges support from ERC 740943 GeoBrown. 

\section{Preliminaries}\label{s.prel}
\subsection{Planar maps}
In this section, we present the elementary concepts that appear in this work. For an introduction to planar maps, we recommend \cite{GJ04, FS09,BM11}.

A \textsc{rooted planar map}, or \textsc{map} for short, is a finite connected graph embedded in the sphere that has a marked oriented edge. We consider that two embeddings as the same map if there is an  homeomorphism between them that preserves the orientation (i.e. respecting a cyclic order around every vertex). We call \textsc{root edge} the marked oriented edge and \textsc{root vertex} its starting point. We denote $v_0$ the root vertex.

A map $\mm_1$ is said to be a submap of $\mm_2$ (with the notation $\mm_1\subset_M \mm_2$) if $\mm_1$ can be obtained from $\mm_2$ by suppressing edges and vertices. This definition implies that $\mm_1$ respects the cyclic order of $\mm_2$ in the vertices and edges remaining. A decorated map is a map with a special submap.

The \textsc{faces} of a map are the connected components of the complement of the edges in the embedding. The \textsc{degree} of a face is the number of oriented edges for which the face lies at its left. In this work, we only work with \textsc{quadrangulations} where all faces (except maybe on one) have degree equal to $4$. 

The face to the left of the root edge is called the \textsc{root face}.  In what follows, \textsc{maps with a boundary} are maps, $\mm^b$, where the root face plays a special role: it has arbitrary degree $2m$. The set of oriented edges that have the root face to its left are called the \textsc{boundary}. The number of oriented edges in the boundary will be called its size. The boundary will be seen  as $\gb:\S^1\mapsto \mm^{b}$ a cyclic path of vertices mapping each of the $2m$ roots of the unity to each vertex appearing in counter-clockwise sense starting on the root vertex. All faces different from the root face are called \textsc{internal faces}.

When the boundary of the map is \textsc{simple}, i.e., the boundary is not vertex-intersecting, the curve $\gb$ is a bijection. We call the label of a vertex $v$ of the boundary the appearance number while going counter-clockwise on the boundary starting by the root vertex. We also label the boundary edges as the label of the vertex where they start from. 

A \textsc{rooted plane tree} of size $m$, or \textsc{tree} for short, is a planar map with only one face and $m$ edges. We will encode plane trees using walks. In the literature, one can find several of these codings, see for example \cite[Section 1]{LG05}. Here we are interested in the contour function. This is a bijection that associates to each rooted plane tree with $m$ edges a Dyck path $C$ indexed by $\cro{0,2m}$. For a more detailed description we refer to \cite[Section 1.1]{LG05} and \cite[Section 2]{Ber07}.

A tree $\tt$ has an intrinsic way of visiting all of its oriented-edges. This visit can be represented by a cyclic path $\gtt:\S^1\to \rtt$ that represents a walker that starts from the root vertex and turns around the tree (see Figure \ref{f.conttree}) associating with each of the $2m$ roots of unity the vertices that he visits. The walker follows the direction of the root edge touching the tree with his left hand\footnote{Note that, in the literature, the walker usually walks following its right hand. In this work, the left hand convention simplify some statements.} as long as it walks. The walker, then, continues until he returns to the root edge. Note that this walk visits every oriented edge only once. Now, we define the contour function of $\tt$ as $C^\tt : \cro{0,2m} \to \N$ the function for which $C^\tt(n)$ is the distance to the root vertex (height) of the vertex visited at time $n$ by the walker (time 0 for the root vertex).  In this case, the inverse is given by the pseudo-distance in $\cro{i,j}$
\begin{align}\label{e.distance_C}
	\dd_C(i,j) = C(i)+C(j)-2\min_{\ell \in \cro{i,j}} C(\ell) .
\end{align}

\begin{figure}[!h]
	\centering
	\includegraphics[scale=1]{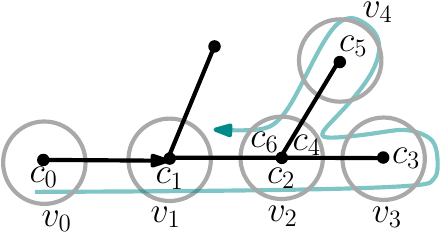}
	\caption{{\small Tree with part of the contour in green, the corners are numbered as $c_i$ and are shown in gray. All the corners belonging to the same circle belong to the same equivalence class associated with a vertex, for example $v_2 = [c_2]_c= [c_4]_c= [c_6]_c$. }}
	\label{f.conttree}
\end{figure} 

Finally, we define a \textsc{tree-decorated map} as a pair $(\mm, \ttm)$ where $\mm$ is a map (without a boundary) and $\ttm\subset_{M}\mm$ is a tree of size $m$. In this work we will study tree-decorated quadrangulations $(\rqq, \ttm)$, which are tree-decorated maps where $\rqq$ is a quadrangulation. Furthermore, for this work we require that the root edge of $\ttm$ coincides with the edge-root of $\rqq$. This allows us to represent $(\rqq,\ttm)$ as a pair $(\rqq,\gm)$, where $\gm:\S^1\to \ttm\subseteq_{M} \rqq$ is the cyclic path that starts from the root and represents the walker following the contour of the tree $\ttm$ (again it associates to the $2m$ roots of the unity the vertices appearing on the path starting from the root vertex and following the root edge). 

In words, when the objects are decorated in our setting they define curves implicitly as explained before, i.e. $\gamma:E \to M$, with $E$ equal to $\S^1$ or $\R$ plus the condition that the curve extends continuously to $[-\infty,\infty]$. 
\begin{rem}\label{cont_ext}
	All the curves previously introduced are curves from discrete spaces, but we can transform them into curves of continuous spaces as follows. We linearly interpolate the discrete graph metric spaces along the edges by identifying each edge with a copy of the interval $[0,1]$. We extend the graph metric in such a way that a path $\gamma$ in $\cro{a,b}$ is linearly interpolated to $[a-1,b]$ by traversing each edge $\gamma(i)$ in the path at unit speed during $[i-1,i]$.
\end{rem}

Rigorously for a metric space $(M,d)$, we consider $C_0(E,M)$ the space of continuous curves $\gamma:E\to M$. Notice that each curve defined on an interval $[a,b]$ can be seen as an element of $C_0(\R,M)$ by considering $\gamma(t)=\gamma(a)$ for $t<a$ and $\gamma(t)=\gamma(b)$ for $t>b$. To compare two curves $\gamma_1,\gamma_2 \in C_0(E,M)$ we use the uniform metric $\dd_{U}$ defined as
\[
\dd_{U}(\gamma_1,\gamma_2) = \sup_{t \in E}d(\gamma_1(t),\gamma_2(t))
\]

In the case of trees, we decorate the metric space by $\gtt$. For the case of maps with a boundary, we decorate them by $\gb$ the curve that starts from the root and follows the boundary at constant speed. In the case of tree-decorated maps, we decorate-them with the ``Peano''-type curve $\gm$ associated to the contour exploration of the tree in the map. In all of these cases $E=\S^1$. 

In this work, we also need to work in cases where $E=\R$. In these cases, we need to work with the truncation of the curve. To do define this truncation first take $r>0$, and define 
\begin{align}
	\underline{\tau}_r^\gamma = (-r)\vee \sup\{t<0: d(\gamma(0),\gamma(t))=r\} \; \text{ and }\; \overline{\tau}_r^\gamma = (-r)\wedge \inf\{t<0: d(\gamma(0),\gamma(t))=r\}
\end{align}
where $d$ is the metric associated to the metric space where $\gamma$ is embedded in. Then, the $r$-truncated of $\gamma$ is the curve 
\[
\gamma_r(t) = \begin{cases}
	\gamma(\underline{\tau}_r^\gamma) &\text{ if } t<\underline{\tau}_{r}^\gamma\\
	\gamma(t) &\text{ if } t\in[\underline{\tau}_{r}^\gamma,\overline{\tau}_{r}^\gamma]\\
	\gamma(\overline{\tau}_r^\gamma) &\text{ if } t>\overline{\tau}_{r}^\gamma.
\end{cases}
\] 

We denote by $\Ball{M,u}{r}$ the ball of radius $r$ centred at $u$ in $M$ and we set $\Ball{M}{r}$ as the ball centred at the root vertex. We also define the $r$-truncation of $(M,\gamma)$ as the curve-decorated metric space $\Tr{M,\gamma}{r} := (\Ball{M}{r}, \gamma_r)$, where the metric in $\Ball{M}{r}$ is the infimum with respect to the metric of $M$ over paths completely contained on $\Ball{M}{r}$.

Finally, let us make a remark regarding the notation. At the discrete level, we will only work with quadrangulations and because of this the notation of some of the elements associated to them have a superscript $\qq$. Since in the scaling limit settings these maps converge to metric spaces $M$ that have no ``quadrangular'' nature, we change the superscript $\qq$ by $M$ in the notation of the continuum objects. Additionally, in our notation to disambiguate some cases we use lower case letters for discrete object and upper case letters for continuous objects.

\subsection{The tree-decorated map and its bijection}\label{ss.tdm}
In this subsection, we will describe the so-called gluing procedure introduced in \cite{FS19}. We direct the reader to Section 3 of that paper to see the proofs and details. Here we give a short summary of the bijection.

Take a couple $(\qq^b,\tt)$ of a quadrangulation with a simple boundary of size $2k$ and $f$ faces and a tree of size $k$. We want to construct $(\qq,\ttm)$ a tree-decorated quadrangulation with $f$ faces and a tree of size $k$. 

Recall that the vertices of the external face of $\qq^b$ are indexed from $0$ to $2k-1$, and call $C$ the contour function of $\tt$. The function $C$ induces an equivalent relation on vertices via the zeros of equation \eqref{e.distance_C}, and define $V'$ as the set of equivalence classes.

Let us now construct $\qq$. The vertex set is made by the union of $V'$ with all vertices of $\qq^b$ that do not belong to the exterior face. The edge set of $\qq$ is constructed from those of $\qq^b$ in the following way. Let $(x,y)$ be an oriented edge of $\qq^b$, then the edge $(G(x),G(y))$ is in $\qq^b$, where
\begin{equation}
	G(x):= \begin{cases}
		[l]&\text{ if } x\in V',\\
		x  &\text{else,} \end{cases}
\end{equation}
where $l$ is the label of $x$ in the boundary, and $[l]$ is the equivalence class of $l$ under the equivalence relation defined by $C$.

\begin{figure}[!ht]
	\begin{subfigure}{0.4\textwidth}
		\centering
		\scalebox{1}[1]{\includegraphics[scale=0.4]{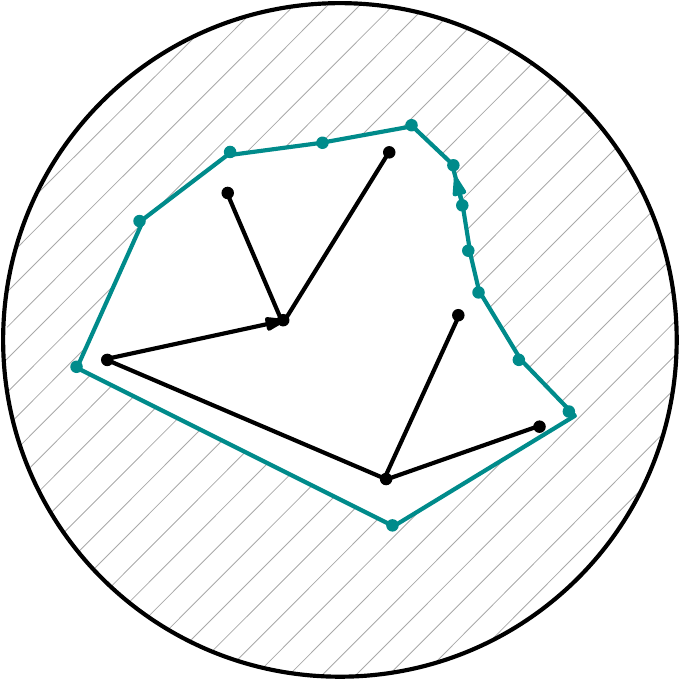}}		
	\end{subfigure}\hspace{-0.05\textwidth} $\substack{\text{Gluing}\\\longrightarrow}$ \hspace{-0.05\textwidth}
	\begin{subfigure}{0.4\textwidth}
		\centering
		\scalebox{1}[1]{\includegraphics[scale=0.4]{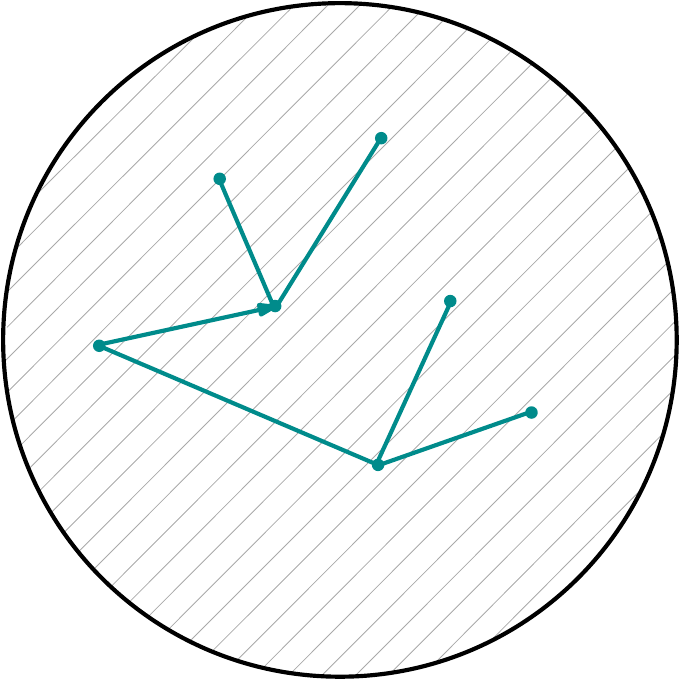}}
		
	\end{subfigure}
	\caption{A simple sketch of the gluing procedure.}
	\label{dectoboundandtree}
\end{figure}

\begin{rem}
	Note that because the boundary of $\qq^b$ is simple, the vertices $V'$ have the same tree-structure in $\qq$ as in $\tt$.
\end{rem}

\subsection{Topologies and convergences}
Now, we describe the topologies that will be used throughout this paper. In this work, we will deal with two types of limits: local limits and scaling limits. For the local limits we use the Benjamini-Schramm topology, and for the scaling limits we use the Gromov-Hausdorff topology in the compact case and the local Gromov-Hausdorff in the non-compact case. When object are decorated, we will consider them decorated by curves, and therefore we will use a strengthen version of these topologies. 

\subsubsection{The Benjamini-Schramm Uniform topology}
The Benjamini-Schramm uniform topology allows us to describe the local limit of a sequence of curved decorated maps $(\mm_n,\gamma_n)$. We say that $(\mm_n,\gamma_n)$ converges to $(\mm,\gamma)$ if for any $R>0$, we have that
\begin{align*}
	\Tr{\mm,\gamma}{r} = \Tr{\mm_n,\gamma_n}{r}
\end{align*}
for all $n$ big enough.

This topology is metrizable, nevertheless we will not need the exact description of its metric.

The (undecorated) Benjamini-Schramm topology arises as the decorated Benjamini-Schramm topology when the curve is constant equal to the root.

\subsubsection{The Gromov-Hausdorff Uniform topology}\label{sec:GH}
This topology will be used to describe the limit of curve-decorated sequence of maps $(\mm_n,\gamma_n)$ in the compact case. We define the following distance between two decorated metric spaces $(\mm_1,\gamma_1)$ and $(\mm_2,\gamma_2)$,
\begin{align*}
	\dd_{GHU}((\mm_1,\gamma_1),(\mm_2,\gamma_2))= \inf_{\phi_1,\phi_2}\{ \dd_{Haus}(\phi_1(\mm_1),\phi_2(\mm_2)) +\dd_{Unif}(\phi_1(\gamma_1),\phi_2(\gamma_2)), \}
\end{align*}
where the infimum is taken over all isometries $\phi_1$ and $\phi_2$ taking $\mm_1$ and $\mm_2$ (respectively) to a common metric space. Here, the Haussdorf distance $\dd_{Haus}$ is a distance between closed set of a given metric space $E$ and is defined as
\begin{align*}
	\dd_{Haus}(C,C')= \inf\{\epsilon>0: C\subseteq C'+B_\epsilon(E), C'\subseteq C+ B_\epsilon(E) \}.
\end{align*}
The (undecorated) Gromov-Hausdorff topology appears in this context as the decorated Gromov-Hausdorff topology when the curve is constant equal to the root.

\begin{rem}[\textbf{GHU Compactness}]\label{compGHU} From Lemma 2.6 in \cite{GM3}, if a set $S$ satisfies: $S$ compact in the Gromov-Hausdorff sense and $S$ equicontinuous, then $S$ is pre-compact in the Gromov-Hausdorff Uniform topology. Here equicontinuous applies for the curves; and of course it depends on $E$ and its topology.
\end{rem}

\subsubsection{The Local Gromov-Hausdorff uniform topology}\label{sec:LGHU^}
This topology will be used to describe the limit of curve-decorated sequence of maps $(\mm_n,\gamma_n)$ in the non-compact case. We define the following distance between two decorated metric spaces $(\mm_1,\gamma_1)$ and $(\mm_2,\gamma_2)$,
\begin{align*}
	\dd_{LGHU}((\mm_1,\gamma_1),(\mm_2,\gamma_2))= \int_0^\infty e^{-r}(1\wedge d_{GHU}(\Tr{\mm_1,\gamma_1}{r},\Tr{\mm_2,\gamma_2}{r}))dr
\end{align*}

The (undecorated) local Gromov-Hausdorff topology appears in this context as the decorated Gromov-Hausdorff topology when the curve is constant equal to the root.

\subsection{Infinite trees} 

We discuss now about the limit of plane trees in different topologies. The first limit that one can study is the local limit. One object that arise naturally as this type of limit is the infinite critical geometric tree.

\begin{defn}\label{p.description_tree}
	The infinite critical geometric tree $\rtt_\infty$ is defined by the following construction
	\begin{enumerate}
		\item Take a copy of the graph given by the natural numbers $\N$, this is called the spine. We denote the elements of the spine $\tau_{-n}$.
		\item Associate to each vertex $v$ of the spine two critical geometric GW trees and hang one to the positive and one to the negative half-plane with $v$ as the root of both trees.
		\item Root the tree in the edge $\overrightarrow{\tau_{0}\tau_{-1}}$.
	\end{enumerate}
\end{defn}

The way to obtain this object as limit is presented in the following theorem.
\begin{thm}[Lemma 1.14 \cite{K86} and Proposition 5.28 \cite{LP17}]\label{t.ICT} Let $\rtt_m$ be a uniform tree with $m$ edges. One has that $\rtt_m\rightarrow \rtt_\infty$ in law for the Benjamini-Schramm topology as $m\rightarrow \infty$. The resulting random object $\rtt_\infty$ is an a.s. one-ended tree called the \textit{infinite critical geometric tree}. 
\end{thm}

\begin{defn}\label{d.subtreee}
	For an infinite critical geometric tree $\tt_\infty$ and $m\in \R$ we define $\tt_\infty(m)$ the finite tree created by all the vertices of the spine that are at distance smaller than or equal to $m$ to the root and all the trees attached to them. 
\end{defn}

It is useful to us to construct tress using the so-called contour function.
\begin{defn}[Real trees] Let $f: I\subseteq \R \to \R$ be a continuous function. For $s,t\in I$ we define the following pseudo-distance
	\[
	\dd_f(t_1,t_2)= f(t_1)+f(t_2)-2\inf_{s \in \cro{t_1,t_2}}f(s).
	\]
	We define the tree coded by $f$ as the metric space $(\mathcal{T}_f,\dd_f)$ consisting of $I/\{\dd_f=0\}$ and we associate as the root the equivalent class of $0$ under $\dd_f$. Notice that $\dd_f$ in this space induces a metric (that we also call $\dd_f$).
\end{defn}
We can use this to write the infinite critical geometric tree by means of a simple random walk\footnote{In this paper, all discrete functions will be interpolated linearly so they are continuous.}.
\begin{prop}\label{prop:rwdes}
	Let $X^+$ and $X^-$ be two independent simple random walks taking values $0$ at $0$. Define
	\begin{align*}
		X(t)=\begin{cases}
			X^{+}(t) & \text{ if } t\geq 0,\\
			X^{-}(-t) & \text{ if } t< 0.
		\end{cases}
	\end{align*}
	Then, $(\TT_X,\dd_X)$ is the infinite critical geometric tree\footnote{Here we represent the infinite critical geometric tree as a metric space where the edges are replaced by copies of $[0,1]$.}.
\end{prop}

First of all notice that the pseudo-distance gives an infinite tree with a unique infinite branch. To show that the distribution of the simple random walk description coincides with the one given in \Cref{p.description_tree} we use the decomposition of the simple random walk by records. Starting from $0$ consider $X_t$ up to the first time $\tau_{-1}$ it hits $-1$. It is well known\footnote{Consult \cite[Prop. 1.5.]{LG05} together with the bijection between the Lukasiewicz function and the contour function of a tree.} that $\P(\tau_{-1}= 2n+1)$ is equal to the probability of a critical geometric GW tree having size $n$. We define for $i> 2$ the records as the fists time $\tau_{-i}$ that $X_{\tau_{-i-1}+t}$ hits $-i$. We identify that the pseudo-distance associate to each record of the simple random walk a tree in the positive half-plane with the same distribution as the critical geometric GW tree. It is clear that one can do the same for the negative part giving trees in the negative half-plane distributed as the critical geometric GW tree.

Now, we introduce the bi-infinite critical geometric tree, which can roughly be seen as unzipping the infinite critical geometric tree along the spine.
\begin{defn}[]\label{p.description_tree_2}
	The bi-infinite critical geometric tree $\rtt_{-\infty}^\infty$ is constructed as follows 
	\begin{enumerate}
		\item Take a copy of the graph given by the integer numbers $\Z$, we called it the bi-infinite spine. We denote the element associated to $n \in \Z$ in the bi-infinite spine $\tau_{-n}$. 
		\item Associate to each vertex $v$ of the spine one critical geometric GW tree and hang it in the positive half-plane with respect to the line $\Z$.
		\item Root the tree in the edge $\overrightarrow{\tau_{0}\tau_{-1}}$.
	\end{enumerate}
\end{defn}	

\begin{defn}\label{d.subtreeebi}
	For a bi-infinite critical tree $\tt_{-\infty}^\infty$ and $m\in \R$ we define $\tt_{-\infty}^{\infty} (m)$ the finite tree created by all the vertices of the (negative and positive) spines that are at distance smaller than or equal to $m$ to the root and all the trees attached to them. 
\end{defn}
Again we can construct the bi-infinite critical geometric tree by means of a simple random walk as follows.	
\begin{prop}
	Let $Y^+$ and $Y^-$ be two independent simple random walks started at height 0 with $Y^+$ conditioned to be positive. Define
	\begin{align*}
		Y(t)=\begin{cases}
			Y^{+}(t) & \text{ if } t\geq 0,\\
			Y^{-}(-t) & \text{ if } t< 0.
		\end{cases}
	\end{align*}
	Then, $(\TT_Y,\dd_Y)$ is the bi-infinite critical geometric tree.
\end{prop}
This follows from the same idea as in the deduction of \Cref{prop:rwdes} with the difference that records in the negative part has to be taken as the last time the random walk escapes each positive level.

Another way to try to understand how big a uniform tree looks like is through a renormalisation of the distance of the tree, so that its diameter remains of constant order. We do this by dividing the distance of $\rtt_m$ by $\sqrt{m}$ leading to a limiting object which is continuos and has finite volume called the continuum random tree (CRT).
\begin{defn}
	Let $(\e_t : t\in[0,1])$ be a Brownian excursion. The CRT, $\Crt$, is the random tree (metric space) defined by $(\TT_{\e},\dd_{\e})$.
\end{defn}

The image of the Lebesgue measure of the map from $[0,1]$ to $\TT_\e$ gives a natural parametrisation to explore the contour of the real tree at unit speed and formalises the length of the CRT. In the sequel we consider the contour curve $\gtt$ parametrised in such a way. As a consequence if $\TT$ has length 1, then $\sigma\TT$ has length $\sigma$.
\color{black}

The following theorem formalises the renormalisation of the finite volume limit.
\begin{thm}[Theorem 8 \cite{AL91}]\label{t.CRT}
	Let $\rtt_m$ be a uniformly chosen tree with $m$ edges and consider it as a metric space with its natural graph distance. Then $(2m)^{-1/2} \rtt_m$ converges in law to the CRT $\Crt$ for the Gromov-Hausdorff topology.
\end{thm}

Another renormalization technique is used to obtain continuous limits with infinite volume, this is obtained when the size and the diameter tends to infinity in a suitable way.
\begin{defn}
	Let $W^+$ and $W^-$ be two independent standard Brownian motions started at 0. We define the process $(W(t): t\in \R)$ as 
	\begin{align*}
		W(t)=\begin{cases}
			W^{+}(t) & \text{ if } t\geq 0,\\
			W^{-}(-t) & \text{ if } t< 0.
		\end{cases}
	\end{align*}
	The Infinite continuous random tree (ICRT) $\Rtt_\infty$ is defined as the random tree $(\TT_W,\dd_W)$.
\end{defn}
The next result says how we can obtain the ICRT from a discrete tree.
\begin{prop}[Theorem 11 ii)\cite{AL91}] \label{t.ICRT}
	Let $\rtt_m$ be a uniformly chosen tree with $m$ edges and consider it as a metric space with its natural graph distance. Consider also any sequence $(k_m:m\in\N)$ satisfying $k_m\to\infty$ and $k_m  m^{-1/2}\to 0$. Then $ k_m^{-1}\rtt_m$ converges in law to the metric space ICRT for the Local Gromov-Hausdorff topology.
\end{prop}
Again, it is possible to obtain a description by means of a random walk.
\begin{defn}\label{def:BICRT}
	Let $Z^+$ and $Z^-$ be two independent standard Brownian motions started at 0 with $Z^{-}$ conditioned to stay positive (i.e. $Z^{-}$ has the law of a Bessel-3 process). We define the process $(Z_t: t\in \R)$ as 
	\begin{align*}
		Z(t)=\begin{cases}
			Z^{+}(t) & \text{ if } t\geq 0,\\
			Z^{-}(-t) & \text{ if } t< 0.
		\end{cases}
	\end{align*}
	The bi-infinite continuous random tree $\Rtt_{-\infty}^\infty$ is the random tree defined by  $(\TT_Z,\dd_Z)$.
\end{defn}

\subsection{Infinite quadrangulations with a boundary}
We present here the limiting objects of quadrangulations with a boundary in different topologies.
Again we start with the local limit.
\subsubsection{$\UIHPQ$ and its peeling}\label{ss.UIHPQ}
The uniform infinite half-plane quadrangulation with simple boundary ($\UIHPQ$) is the local limit of a well-chosen quadrangulation with a simple boundary. In this section, we will shortly present it and describe its Markov property.

Let $\qq^b_{f,m}$ be a uniformly chosen element in $Q^b_{f,m}$ the set of planar quadrangulation with a simple boundary of size $2m$ and $f$ internal faces. It is easy to see that there exists a random variable $\qq^{b}_{m}$ on infinite quadrangulation with a simple boundary of size $2m$, such that $\qq^b_{f,m}$ converges to $\qq^{b}_{m}$ as $f$ goes to $\infty$ for the Benjamini-Schramm topology. We define the $\UIHPQ$, $\qq$, as the law on random quadrangulation with a simple boundary of infinite size as  that arises as the limit in law of $\qq^{b}_{m}$ as $m$ grows to $\infty$ (see \cite{CM12,CC18}). This limit is one ended, in the sense that if one takes out one square of $\qq$ the complement of the graph has only one infinite connected component. In this limit the curves $\gb_m$ coding the boundary of $\qq^{b}_{m}$ also converge to a limit $\gb$ which is a parametrisation of $\Z$ and therefore the convergence holds in the Benjamini-Schramm Uniform topology. Furthermore, if one shifts the root-edge in the boundary to another boundary point, the law of the resulting map is the same as the original one \cite{CC18}, this property is called invariance under rerooting.

The $\UIHPQ$ $\qq$ satisfies an interesting Markov property. Assume that one conditions on all the quadrilaterals $Q^r\in \qq$ that contain the root vertex $r$ of $\qq$. Then, the unbounded connected component of $\qq\backslash Q^r$ (rooted properly) also has the law of a $\UIHPQ$.

\begin{figure}
	\includegraphics[scale=0.8]{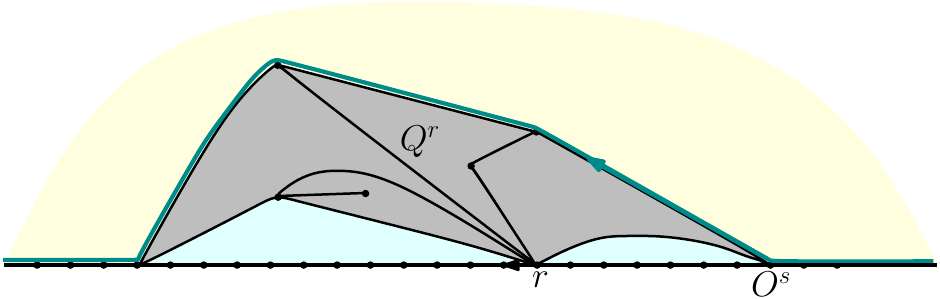}
	\caption{Markov property sketch. In a $\UIHPQ$ conditioning on $Q^r$, the light blue parts are given by Botzmann quadrangulations conditioned to their boundary size and the yellow part is again a $\UIHPQ$ we put the cyan line to represent the simple boundary of the yellow part.}
\end{figure}

For the $\UIHPQ$, we associate its boundary vertices with $\Z$ and its root edge to the one going from $0$ to $-1$. Let us define the simple overshoot $O^s(0)$ from the vertex $0$. To do this, we take all edges connected to $0$ and we look for the ones that intersect $\Z^+$, the over-shoot is then taken as the maximum value of that intersection (and it is $0$ if not). One has that \cite{CC19}
\begin{align} \label{overshoot:proba}
	\P(O^s(0)\geq k)\asymp k^{-3/2},\quad \text{ as } k\nearrow \infty.
\end{align}
We can define, now, the \textsc{(infinite) overshoot} from $0$. The (infinite) overshoot $O(0)$ is the biggest positive $z$ such that there is a face containing $z$ and a vertex in the negative boundary (it is $0$ if there is no such an edge). By a summation over $O^s(-j)$ one has that \cite{CC19}
\begin{align*}
	\P(O(0)\geq k)\asymp k^{-1/2},\quad \text{ as } k\nearrow \infty.
\end{align*}

\subsubsection{Brownian half-plane} The Brownian half-plane (BHP) arises as the limit in distribution of the $\UIHPQ$ in the Local Gromov-Hausdorff topology. More formally $\lambda \cdot \UIHPQ$ converges in law in the Local Gromov-Hausdorff to the $BHP$ as $\lambda \to 0$ \cite[Theorem 3.6]{BGR19}. 

\begin{prop}\label{prop:SM}
	The Brownian half-plane is invariant under rerooting and this operation is strong mixing.
\end{prop}
\begin{proof}[Proof of \Cref{prop:SM} ]
	Invariance under rerooting is inherited from the $\UIHPQ$ and the strong mixing property is a consequence of the Markov property (on filled-in balls with target point at infinity) of the Brownian half-plane \Cref{prop:markovBHP} and its invariance under rerooting. More precisely, combining these properties and following the same lines as in \cite[Lemma 2]{CC19} one obtains that thanks to the Markov property the balls around the root and the remainder of the map (which is distributed also as a Brownian half-plane properly rooted, thanks to the invariance under rerooting) are asymptotic independent as the distance between the two roots goes to infinity, this asymptotic independence let us conclude the strong mixing condition.
\end{proof}

\subsubsection{Brownian disk} The Brownian disk appears as the scaling limit of a planar map with a boundary of appropriate size. This is done in the following theorem  which is a slight improvement of the main result of \cite{BCFS} whose proof can be found in Section \ref{a.b}.
\begin{thm}\label{t.BD} Fix $\sigma\in \R^+$. A sequence $(\rqq_{f,\sigma \sqrt{f}}^b,\gb)_{f\in \N}$ of uniformly chosen quadrangulations with a simple boundary of size $2\lfloor\sigma \sqrt{f}\rfloor$ and $f$ internal faces . Then, $((9f/8)^{-1/4}\rqq_{f,\sigma \sqrt{f}}^b,\gb)_{f\in \N}$ converges in law to the (decorated) Brownian disk  $(\QQQ_{3\sigma},\Gb)$ with perimeter $3\sigma$ and area 1 for the Gromov-Hausdorff uniform topology. Furthermore, $\QQQ_{3\sigma}$ has the topology of the disk.
\end{thm}
This theorem was first proved for the uniform case when the boundary is not simple in \cite{BetM} and also in the case of Free Boltzmann quadrangulations with simple boundaries in \cite{GM} and then generalized to the case of uniform quadrangulation with simple boundary in \cite{BCFS}. The topology of the limit was first described in Section 2.3 of \cite{BetM}, and it can be shown that a.s. $\Gb(t) \in \partial \QQQ_{3\sigma}$ for all $t\in \S$.

\subsection{Definition of the shocked-map}\label{SMdef}
Here we introduce the candidate for the scaling limit of the critical tree decorated map by doing the analogue of the bijection of Section \ref{ss.tdm} in the continuum setting.

\begin{defn}[Shocked map]\label{d.sm}
	Let $\sigma>0$, $(\QQQ_{\sigma},\dd^{BD})$ a Brownian Disk with boundary of size $\sigma$ and let $\Crt$ be an independent CRT. Take $\Gb: \S^1 \to \partial \QQQ_{\sigma}$  the continuous curve that visits  $\partial \QQQ_{\sigma}$ at unit speed starting at the root edge and $\Gtt: \S^1 \to \Crt$ the contour exploration\footnote{The contour exploration associated to the CRT is the curve generated by the image of the identity function in $[0,1]$ under the glueing of the Brownian excursion used to create it.} of $\Crt$. The shocked map of size $\sigma$ is the (curve-decorated) metric space $(\Smm_\sigma,\dd^{SM},\Gs)$, obtained by starting with $(\QQQ_{\sigma},\dd^{BD})$ and identifying all points in $x,y \in  \partial \QQQ_{\sigma}$ such that
	\begin{align*}
		\Gtt\circ(\Gb)^{-1}(x)=\Gtt\circ(\Gb)^{-1}(y).
	\end{align*}
	Here $\Gs$ is the curve defined as the image of $\Gb$ under this identification.
\end{defn}
Let us give another equivalent description of the shocked map in the usual language of metric spaces defined as equivalent classes of pseudo-distances.
\begin{rem}
	To identify $\QQQ_{\sigma}$ using $\Gb$ and $\Gtt$, we define the pseudo-distance
	\[
	\dd^{SM}(x,x')= 
	\inf\Big\{\sum_{i=0}^{k}\min \{\dd_{D}(x_i^b,y_{i}^b) \}
	\Big\},
	\]
	where the infimum is taken over all $k\geq 0$, sequences
	$t_0,s_1,t_1,s_2,t_2,\dots,s_k\in \S^1$, such that $x_0^b=x$, $y_k^b = x'$ and for all other $i\in \cro{0,k}$
	\begin{align*}
		x_i^b:= \Gb(s_i)\text{ and }y_i^{b} := \Gb(t_i),
	\end{align*}
	such that $\Gtt(t_i)=\Gtt(s_{i+1})$. Then, $(\Smm_\sigma,\dd^{SM},\Gs)$ is the (curve-decorated) metric space where $\Smm_\sigma$ is given by $[0,1]/\{\dd^{SM} = 0\}$.
\end{rem}

Again we can define the infinite volume version of this object

\begin{defn}[Infinite shocked map]\label{d.sminf}
	Let $(\HHH_{\infty},\dd_{\HHH_\infty})$ be a Brownian half-plane and let $(\Rtt_\infty,\dd_\Rtt)$ be an ICRT. Take $\Gb: \R \to \partial \HHH$  the continuous curve that visits  $\partial \HHH_\infty$ at unit speed starting at the root edge and $\Gtt: \R \to \Rtt_\infty$ the contour exploration of $\Rtt_\infty$. The infinite shocked map is the (curve-decorated) metric space $(\Smm_\infty,\dd_{\Smm_\infty},\Gs)$, obtained by starting with $(\HHH_{\infty},\dd_{\HHH_\infty})$ and identifying all points in $x,y \in  \partial \HHH_\infty$ such that
	\begin{align*}
		\Gtt\circ(\Gb)^{-1}(x)=\Gtt\circ(\Gb)^{-1}(y).
	\end{align*}
	Here $\Gs$ is the curve defined as the image of $\Gb$ under this identification.
\end{defn}

We will use the bi-infinite version as it appears as an ``intermediate'' object for our proof. 
\begin{defn}[Bi-infinite shocked map]\label{d.smbinf}
	Let, $(\HHH_{\infty},\dd_{\HHH_\infty})$ be a Brownian half-plane and let $(\Rtt_{-\infty}^\infty,\dd_\Rtt)$ be a bi-infinite continuous random tree. Take $\Gb: \R \to \partial \HHH$  the continuous curve that visits  $\partial \HHH_\infty$ at unit speed starting at the root edge and $\Gtt: \R \to \Rtt_{-\infty}^\infty$ the contour exploration of $\Rtt_{-\infty}^\infty$. The bi-infinite shocked map is the (curve-decorated) metric space $(\Smm_{-\infty}^{\infty},\dd_{\Smm_{-\infty}^\infty},\Gs)$, obtained by starting with $(\HHH_{\infty},\dd_{\HHH_\infty})$ and identifying all points in $x,y \in  \partial \HHH_\infty$ such that
	\begin{align*}
		\Gtt\circ(\Gb)^{-1}(x)=\Gtt\circ(\Gb)^{-1}(y).
	\end{align*}
	Here $\Gs$ is the curve defined as the image of $\Gb$ under this identification.
\end{defn}

\begin{rem}
	The meaning of ``intermediate'' comes from the fact that the boundary of the bi-infinite shocked map can be seen as a copy of $\R$ such that when identifying the part associated with $\R^+$  and $\R^{-}$, one gets the infinite shocked map.
\end{rem}

\section{Infinite continuous volume}\label{s.ic}
In this section we show that two elements at distance $n$ on the spine of the ICRT are mapped by the gluing to two points that are at distance $o(n)$. To be more precise, we consider a Brownian half-plane $\HHH_{\infty}$, an independent infinite continuous random tree $\Rtt_\infty$. Let $\Gh:\R\mapsto \partial \HHH_{\infty}$ be the exploration of the boundary of $\HHH_\infty$ such that $\Gh(0)$ is the root vertex of $\HHH_\infty$ and $\Gtt:\R\mapsto \partial \Rtt_\infty$ be the contour exploration of $\Rtt_\infty$. We define $(\Rmm_\infty,\Gr)$ as the glueing of $\HHH_\infty$ using the equivalence class generated by the curves $\Gh$ and $\Gtt$ as for Definition \ref{d.sm}. The aim of this section is to show that the image of the boundary under the glueing is just one point.
\begin{thm}\label{t.main_section_3}
	The exploration function $\Gr:\R\mapsto \Rmm_\infty$ is constant.
\end{thm}

To show this, it is easier to work with $\Rmm_{-\infty}^\infty$, which is the result of glueing a Brownian half-plane with a bi-infinite CRT $\Rtt_{-\infty}^\infty$. Recall from \Cref{def:BICRT} that $\Rtt_{-\infty}^\infty$ is defined from a process $Z$ which is a Brownian motion started from 0 in the positive axis and a Brownian motion started from 0 and conditioned to be positive (seen backwards) in the negative axis. For simplicity in this section we denote by $\dd_{\Rtt}$ the metric on $\Rtt_{-\infty}^{\infty}$ and for any $y\in \R^+$, we define 
\begin{align}\label{tau}
	&\tau_{y}=\inf\{t\in \R: Z(t)=y  \}
\end{align}

Furthermore, for $y\in \R$,  we define $\kappa_y:= \Gr(\tau_{-y})$ which is the image of the root of the $y$-th tree under the glueing. To show the theorem, we first start by showing that the distance in the infinite branch in the tree-decorated map is a constant times the distance in the tree itself.
\begin{prop}\label{prop1}
	There exists a constant $c \in [0,1]$ such that
	\[
	\frac{\dd_{\Rmm^{\infty}_{-\infty}}(\Phi^{\Rmm^{\infty}_{-\infty}}(\kappa_0),\Phi^{\Rmm^{\infty}_{-\infty}}(\kappa_n))}{n}\xrightarrow[n\to\infty]{a.s.} c.
	\]
\end{prop}

We define the shift $\sh$ on the pair $(\Rtt_{-\infty}^\infty,\HHH_{\infty}):= ((\Rtt_{-\infty}^\infty,\Gtt),(\HHH_{\infty},\Gh))$ as \begin{align*}
	\sh(\Rtt_{-\infty}^\infty,\HHH_{\infty}) = (\sh_1(\Rtt_{-\infty}^\infty),\sh_2(\HHH_{\infty})),
\end{align*} where $\sh_1$ and $\sh_2$ are the shift of $\Gtt$ and $\Gh$ so that they starts in $\tau_{-1}$ instead\footnote{This is equivalent to reroot $\Rtt_{-\infty}^\infty$ and $\HHH_{\infty}$ such that the root is now in $\Gtt(\tau_{-1})$ and $\Gh(\tau_{-1})$ respectively } of $0$.  For the sake of notation \textit{until the end of this section} we  drop the indices $-\infty$ and $\infty$ in the objects appearing in the bi-infinite infinite volume shocked map construction.

\begin{lemma}\label{lem1}
	The shift $\sh$ is a measure preserving transformation. 
	Moreover it is strong mixing, i.e. for every $A, B\in \sigma(\Rtt)\times\sigma(\HHH)$ one has that 
	\[
	\P(\sh^{-n}(A) \cap B) \rightarrow \P(A )\P(B)
	\]
\end{lemma}

\begin{proof}
	We start by proving that this transformation is measure preserving. Consider $f$ and $g$ a bounded $\sigma(\Rtt)$ measurable and $\sigma(\HHH)$ measurable functions respectively. 
	\begin{align*}
		\E(f(\sh_1(\Rtt))g(\sh_2(\HHH)))
		&= \E(\E(f(\sh_1(\Rtt))g(\sh_2(\HHH))| \sigma(\Rtt)))\\
		&= \E(f(\sh_1(\Rtt))\E(g(\sh_2(\HHH))| \sigma(\Rtt)))\\
		&= \E(f(\sh_1(\Rtt)))\E(g(\HHH)).
	\end{align*}
	Where the second equality comes from measurability and the third comes from the fact that knowing $\Rtt$, the shift $\sh_2$ becomes deterministic, and the fact that the Brownian half-plane is invariant under rerooting on the boundary. With this we conclude the independence of $\sh_1(\Rtt)$ and $\sh_2(\HHH)$ and moreover that $\sh_2(\HHH)$ is equal in distribution to $\HHH$.
	
	For the strong mixing property, it suffices to test it in any $\pi$-system $\Pi$ in $\sigma(\Rtt)\times \sigma(\HHH)$. Thus, we take $A=A_1\times A_2$ and $B = B_1\times B_2$ with $A_1,B_1 \in \sigma(\Rtt)$ and $A_2,B_2\in \sigma(\HHH)$.
	Then for every $\varepsilon>0$, there exists $n_0$ such that for all $n\geq n_0$ one has
	\begin{align*}
		\E\left(\mathds{1}_{\sh^{-n}(A)\cap B}\right) 
		&= \E\left(\mathds{1}_{\sh_1^{-n} (A_1)\cap B_1} \E\left(\mathds{1}_{\sh_2^{-n} (A_2)\cap B_2}\Big|\Rtt\right)\right)\\
		&\leq \E\left(\mathds{1}_{\sh_1^{-n} (A_1)\cap B_1} \left(\P\left(A_2\right)\P\left( B_2\right)+\varepsilon \right)\right)\\
		&\leq \left(\P\left(A_1\right)\P\left(B_1 \right)+\varepsilon\right) \left(\P\left(A_2\right)\P\left( B_2\right)+\varepsilon \right)\\
		&\leq 3\varepsilon + \P\left(A\right)\P\left(B \right)
	\end{align*}
	The first inequality follows from the strong mixing property of the Brownian half-plane and the second inequality comes from the strong mixing property  of the bi-infinite continuous Brownian tree\footnote{This comes from the invariance under rerooting together with the Markov property of the Brownian motion.}. The converse inequality (with $-\epsilon$ instead of $\epsilon$) follows along the same lines; with it we conclude.
\end{proof}

We can now prove Proposition \ref{prop1}.
\begin{proof}[Proof of \Cref{prop1}]
	We start  by proving the subadditivity of $g_n(\Rtt,\HHH) = \dd_{\Rmm}(\Phi^{\Rmm}(0),\Phi^{\Rmm}(n))$ for the shift $\sh$.
	\begin{align*}
		\dd_{\Rmm}(\Phi^{\Rmm}(\kappa_0),\Phi^{\Rmm}(\kappa_{n+m})) &\leq \dd_{\Rmm}(0,\Phi^{\Rmm}(n))+\dd_{\Rmm}(\Phi^{\Rmm}(n),\Phi^{\Rmm}(n+m)) \\
		&\leq \dd_{\Rmm}(0,\Phi^{\Rmm}(n))+\dd_{\Rmm}(\Phi^{\Rmm}(\sh^n(0)),\Phi^{\Rmm}(\sh^n(m)))\\
		&\leq g_n(\Rtt,\HHH) + g_m(\sh^n(\Rtt,\HHH)).
	\end{align*}
	This together with \Cref{lem1} applied to the Kingsman's subadditive ergodic theorem which gives the result. The consequence that $c\in [0,1]$ follows since $\dd_{\Rmm}(\Phi^{\Rmm}(0),\Phi^{\Rmm}(n))\leq n$ after the gluing.
\end{proof}
In order to stablish that $\dd_{\Rmm}(\kappa_0,\kappa_n) = o(n)$, we prove that $c$ appearing in \Cref{prop1} is equal to 0.
\begin{prop}\label{prop2}
	The constant $c$ in \Cref{prop1} is equal to 0.	
\end{prop}

The proposition is proven using the following lemma.
\begin{lemma}\label{lem2}
	For every point $u\in \R^+$ one has
	\[
	\frac{\dd_{\Rmm}(\kappa_0,\kappa_u)}{u} \eqd \dd_{\Rmm}(\kappa_0,\kappa_1)
	\]
\end{lemma}
\begin{proof}
	Recall that for $B_t$ Brownian motion and $\ell\in \R^+$ one has that $B_t\eqd B_{\ell^2t}/\ell$ and $\tau_1\eqd \tau_\ell/\ell^2$. We denote by $(H(s,t):\; s,t\in \R)$ the process where $H(s,t) = \dd_\HHH(\Gh(s),\Gh(t))$ for $s,t\in \partial \HHH$;  which has the property that $H(s,t)$ is equal in law to $H(\ell^2 s,\ell^2 t)/\ell$ (this follows from the renormalization applied to the UIHPQ$_S$ to obtain the Brownian half-plane in Theorem 1.12 \cite{GM3}). Also recall the definition of the pseudometric associated to $\Rtt$ and notice that since the contour process of $\Rtt$ in the positive spine is associated to $Z = (Z(t) : t\in\R)$ a standard Brownian motion (\Cref{def:BICRT}), then $\dd_Z(s,t) = \dd_Z(\ell^2 s, \ell^2 t)/\ell$ and we write $s \sim_Z t $ if $\dd_Z(s,t)= 0$.  
	
	We define the set $J(K,u)$ as the set of all sequences $t = (t_i \in \R: i\in \{0,1,\dots,K\})$ of length $K$ such that $t_i\sim_Z t_{i+1}$ and such that $t_0=0$ and $t_K\sim_Z \tau_{-u}$.
	
	We have the following equalities
	\begin{align*}
		\inf_{K\in \N} \inf_{t \in J(K,u)}\Big\{\sum_{i=0}^K \frac{H(t_i,t_{i+1})}{u}\Big\}
		&\eqd \inf_{K\in \N} \inf_{t \in J(K,u)}\Big\{\sum_{i=0}^K H\left(\frac{t_i}{u^2},\frac{t_{i+1}}{u^2}\right)\Big\} \\
		&\eqd\inf_{K\in \N} \inf_{s \in J(K,1)}\Big\{\sum_{i=0}^K H\left(s_i,s_{i+1}\right)\Big\},
	\end{align*} 
	from where we conclude.
\end{proof}

We can now prove the proposition.
\begin{proof}[Proof of \Cref{prop2}]
	From \Cref{prop1} and \Cref{lem2} one has that $c$ is equal in law to $\dd_{\Rmm}(\kappa_0,\kappa_1)$ and since the gluing operation decrease the distances one has that $c\leq \dd_{\HHH}(\kappa^\HHH_0,\kappa^\HHH_1)$, where $\kappa^\HHH_y = \Gh(\tau_{-y})$ and $\tau_y$ is defined in \eqref{tau}. Recalling that $\tau_{y}$ are independent of $\HHH$, we get that for any $\varepsilon>0$, exists $\delta= \delta(\varepsilon)>0$ such that
	\[
	\P(c\leq \varepsilon)\geq \P(\dd_{\HHH}(\Gtt(0),\Gtt(1))\leq \varepsilon) \geq \delta
	\]
	but since $c$ is constant a.s. we conclude that $c=0$.
\end{proof}

We now note that by mixing Lemma \ref{lem2} and Proposition \ref{prop2}, together with noting that the curve $\Gtt$ is continuous in $\Rmm$, we obtain the following corollary.
\begin{cor}\label{cor2}
	For every $y\in \R$ one has a.s.
	\[
	\dd_{\Rmm}(\kappa_0,\kappa_y) = 0
	\]
\end{cor}

Now, we generalise \Cref{cor2} to show that any two points in the decoration have distance equal to 0.
\begin{prop}\label{prop3}
	Almost surely, for any $s,t \in \R$
	\[
	\dd_{\Rmm}(\Gm(s),\Gm(t)) = 0
	\]
\end{prop}

\begin{proof}
	Since the the curve $\Gr$ is continuous, it is enough to show that for any $q,r \in \Q$  almost surely  $\dd_{\Rmm}(\Gtt(q),\Gtt(r))=0$. To do that, let us define $\mathfrak{p}(q\rightarrow +\infty)$ the unique simple path starting from $\Gtt(q)$ that goes to $+\infty$ in $\Rtt$. Note that $\mathfrak{p}(q\rightarrow +\infty) \cap \mathfrak{p}(r\rightarrow +\infty) $ is non-empty, and take $u \in \R$ such that $\Gtt(u)$ is the smallest element in that intersection (in the order given by the path $\mathfrak{p}(q\rightarrow +\infty)$). It is enough to show that $d_{\Rmm}(\Gtt(r),\Gtt(u))=0$. To do that, we use the invariance of the distribution under re-rooting Lemma \ref{lem1}, and we  re-root $\Rmm$ at the point $\Gtt(r)$ and we call this rerooting $\Rmm'$. It is clear now that $\kappa'_0=\Gtt(r)$ and that there exists $u'$ such that $\kappa_{u'}=\Gtt(u)$. We conclude from Corollary \ref{cor2}. 
\end{proof}
We now conclude with the proof of Theorem \ref{t.main_section_3}.
\begin{proof}[Proof of Theorem \ref{t.main_section_3}]
	We note that the distance of the glueing between $(\Rmm_{\infty}, \HHH_\infty)$ is bigger than $(\Rmm_{-\infty}^\infty, \HHH_\infty)$, as one can obtain the first glueing, by first doing the second glueing and the identifying the points $\kappa_{y} $ with $\kappa_{-y}$ for all $y\in \R^+$. We conclude by Proposition \ref{prop3}.
\end{proof}
\section{Infinite discrete volume} \label{s.id}

\subsection{Local limit of the tree-decorated quadrangulation}\label{S.ll}
The objective of this section is to understand what a tree-decorated quadrangulation looks like when both the map and the tree are big. This will be done by obtaining a local limit of the map looked from its root, i.e, the root of the tree.  Let $\qt{f}{m} = (\rqq,\Gr)$ be a pair uniformly chosen in $Q_{f}^{\textsf{T},m}$ the set of pairs with first coordinate a quadrangulation with $f$ faces and second coordinate describing the contour of a tree with $m$ edges which is a submap containing the root edge of the first coordinate.  Even tough our results are expressed by means of a quadrangulation decorated by a curve, we will indistinctly use that they are decorated by a tree, since they are in bijection.

\begin{prop}\label{proploc} As $f\rightarrow \infty$, $\qt{f}{m} $ converges in distribution (for the Benjamini-Schramm Uniform topology) to a limit we call $\qt{\infty}{m}$. Furthermore, as $m\to \infty$ we have that $\qt{\infty}{m}$ converges in distribution (for the Benjamini-Schramm Uniform topology) towards a limit $\qt{\infty}{\infty}$ that we call the infinite tree-decorated quadrangulation (ITQ). In brief,
	\[
	\qt{f}{m}\xrightarrow[local\;(f\rightarrow \infty)]{(d)} \qt{\infty}{m} \xrightarrow[local\;(m\rightarrow \infty)]{(d)} \qt{\infty}{\infty}. 
	\]
\end{prop}
Let us note that this proposition can be extended to other random objects, for example a half-plane tree decorated quadrangulation (see \cite[Chapter 5]{FR19} for a more general statement).
\subsubsection{Description of the local limit}\label{ss.des_lim}
Let $\rtt_\infty$ be an infinite critical geometric tree as in Theorem \ref{t.ICT}, $\gtt_\infty$ its contour curve and let $\rqq^b_{\infty,\infty}$ be a $\UIHPQ$. Note that the vertices that are in the boundary of $\rqq^b_{\infty,\infty}$ can be identified by $\Z$, in a way that the root edge is $\overrightarrow{0(-1)}$ (as the infinite face lies to the left of this edge). We define the infinite tree-decorated quadrangulation (ITQ) $(\rqq_\infty,\Gr_\infty)$ as the graph obtained by taking the quotient (the boundary of) $\rqq^b_{\infty,\infty}$ with the equivalence relationship given by $\gtt_\infty$, i.e., two vertices of $v_1,v_2\in \rqq^b_{\infty,\infty}$ are equivalent if $v_1, v_2 \in \partial \rqq^b_{\infty,\infty}$ and $d_{\rtt_\infty}(v_1,v_2)=0$. The curve $\Gr_\infty$ is the contour of the image of the boundary curve, which is a copy of $\rtt_\infty$ in $\rqq_\infty$.
\begin{figure}[h!]
	\begin{subfigure}{0.6\textwidth}
		\centering
		\scalebox{0.5}[0.5]{\includegraphics{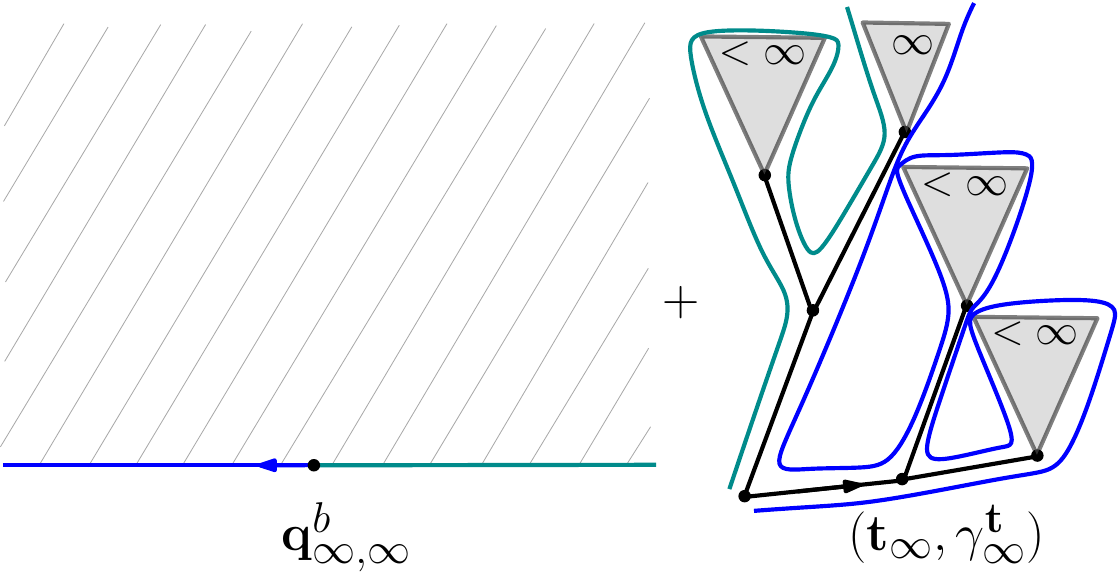} }		
	\end{subfigure}
	\hspace{-0.0\textwidth} $\substack{\text{Gluing}\\\longrightarrow}$ \hspace{-0.0\textwidth}
	\begin{subfigure}{0.3\textwidth}
		\centering
		\scalebox{0.45}[0.45]{\includegraphics{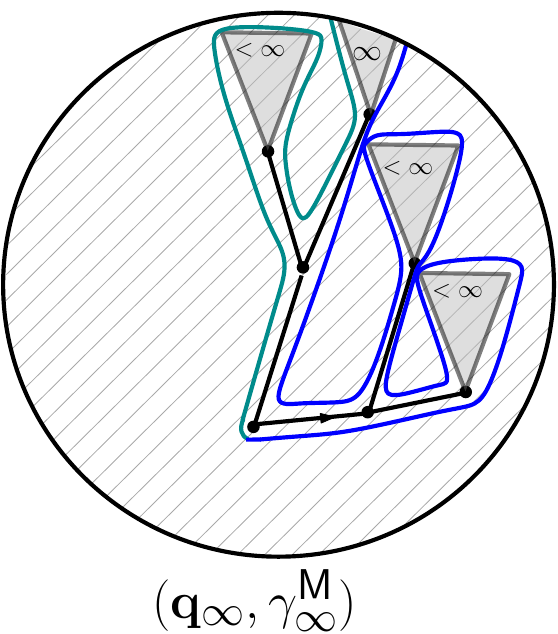}}		
	\end{subfigure}
	\caption{Extended gluing function. We highlight that the right contour in blue and the left contour in cyan are glued to the left and right boundary, respectively. }
\end{figure}

\subsubsection{Proof of the local limit}
Denote the gluing function by 
\[
\phi: \bigcup_{f,m\in \N} Q_{f,m}^{b}\times T_m \to \bigcup_{f,m\in \N} Q_{f}^{\textsf{T},m},
\]
where
\begin{align*}
	Q_{f,m}^{b} &= \text{set of quadrangulations with }f\text{ internal faces and a boundary of length } 2m\\
	T_m &= \text{set of planar trees with }m \text{ edges}
\end{align*}
Let $T_\infty$ be the set of one ended infinite planar trees.

\Cref{proploc} will be a consequence of the following lemma.
\begin{lemma}\label{lem:cont}
	The function $\phi$ admits an extension when $f,m\in \N\cup\{\infty\}$, which is continuous with respect to the product local topology and the Benjamini-Schramm Uniform topology. 
\end{lemma}
This implies \Cref{proploc} by the continuous mapping theorem, \Cref{t.ICT} and the convergence (see \cite{CM12,CC18})
\[
\rqq_{f,m}^b\xrightarrow[local\;(f\rightarrow \infty)]{(d)}\rqq_{\infty,m}^b \xrightarrow[local\;(m\rightarrow \infty)]{(d)} \UIHPQ.
\]
where $\rqq_{f,m}^b$ is a uniform element of $Q^b_{f,m}$.

The natural extension is to glue the infinite boundary starting from the root-edge to the contour of the tree following its root-edge.This extension will never glue the other side of the infinite branch in the one-ended tree, since it will never cross it. This problem can be fix in a simple way, namely to glue edges in both directions starting from the root-edge.
\begin{proof}[Proof of \Cref{lem:cont}]
	Consider $(\qq,\tt)\in  Q_{f,m}^{b}\times T_m$ for $f,m\in \N\cup\{\infty\}$. Define $\overline{\phi}(\qq,\tt)$ as the result of gluing in parallel the left (right) side of the boundary in $\qq$ starting from its root and following its (counter) root-edge sense to the tree starting from its root and following the (counter) root-edge sense. This procedure finishes when the gluing meets from the left and the right (this is well defined since there is an even number of edges).
	This is clearly an extension of $\phi$. For the continuity consider a sequence $(\qq_n,\tt_n)$ converging to $(\qq,\tt)$ in the product local topology. It is easy to see that $\overline{\phi}(\qq_n,\tt_n)$ converges to $\overline{\phi}(\qq,\tt)$ in the local Benjamini-Schramm Uniform topology, since by the locally finite property, for any $R>0$ the ball $\Tr{\overline{\phi}(\qq,\tt)}{R}$ is determined by finite radius balls $\Ball{\qq}{R'}$ and $\Ball{\tt}{R''}$. Since for $n$ big enough $\Ball{\qq_n}{R'} = \Ball{\qq}{R'}$ and $\Ball{\tt_n}{R''}=\Ball{\tt}{R''}$, by applying the gluing procedure we see that $\Ball{\overline{\phi}(\qq_n,\tt_n)}{R}$ and $\Ball{\overline{\phi}(\qq,\tt)}{R}$ coincide for large enough $n$.
\end{proof}

\subsection{Peeling of the infinite tree-decorated quadrangulation}

In this section, we are going to work with the infinite tree-decorated quadrangulation. That is to say, the infinite map defined in Section \ref{S.ll}. We are going to define a specific peeling, i.e. a Markovian way of exploring it. The nice property of the peeling we are going to define is that in its $k$-th step we will have discovered a set that contains the ball of radius $k$ of a given set. The described peeling is based in a closely related peeling that can be found in \cite{CC19,GM2}.

\subsubsection{Description of the peeling}
Let us start with an instance of the $ITQ$, say $(\qq_\infty,\ttm_\infty) = \overline{\phi}(\qq^b_{\infty,\infty},\tt_\infty)$ where $\qq^b_{\infty,\infty}$ and $\tt_\infty$ are instances of the $\UIHPQ$ and infinite critical geometric tree, respectively according to Section \ref{ss.des_lim}. We describe the spine of $\ttm_\infty$ (isometric to $\tt_\infty$) as a copy of $\N$, in which each vertex has two critical GW trees attached to it as in Proposition \ref{p.description_tree} (the one to the left at the top and the right at the bottom). We are going to be interested in peelings of the spine, i.e., $\N$.

Let $r\in \N$, we will study a peeling of the set $\cro{0,r}$ together with all the trees that are attached to those vertices (see Figure \ref{f.peeling_tree}). For simplification let us call this set $\pp_r=\pp$.
\begin{figure}[!h]
	\centering
	\includegraphics[scale=1]{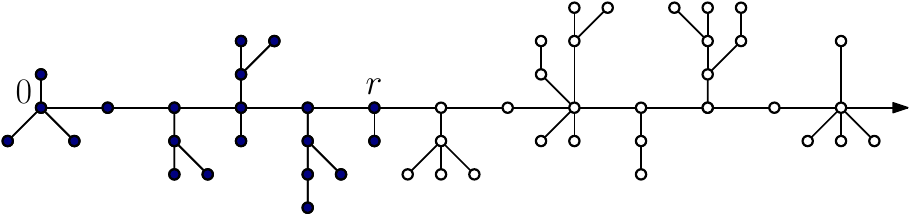}
	\caption{{\small Sketch of the infinite tree $\tt$. The vertices in blue are those belonging to $\pp_r$, when $r=5$. }}
	\label{f.peeling_tree}
\end{figure}

The objective of the peeling will be to construct a sequence of sets $\pp^{(l)} \subseteq \qq_\infty$ such that the ball of radius $l$ around $\pp$ is contained in $\pp^{(l)}$.

In other words, all points in $\qq_\infty\backslash \pp^{(l)}$ should have distance to $\pp$ strictly bigger than $l$.

For each $l$ the peeling process will also define an infinite quadrangulation with infinite simple boundary $\qq^{(l)}\subset_M \qq^b_{\infty,\infty}$ \footnote{This peeling can be read in the preimage, so that $\qq^{(l)}$ is the part to be explored in the preimage.} and an interval $b^{(l)}$ on the boundary of $\qq^{(l)}$. Define $\qq^{(0)}= \qq^b_{\infty,\infty}$, $r^{(0)}=r$ and note that $\pp \subset \qq_\infty$ corresponds (by the gluing) to an interval which we define as $\bb^{(0)}$. Now, iterate for $l\geq 1$ the following
\begin{itemize}
	\item Peel all the faces of $\qq_\infty$ that are the image of a face in $\qq^{(l-1)}$ having a vertex contained in $\bb^{(l-1)}$. Let $r^{(l)}$ be the biggest $n\in \N$ such that $\ttm_\infty$ has a tree attached at $r^{(l)}$ containing a vertex peeled in this step. Note that if we take away from $\qq_\infty$ all the peeled faces and all the vertices associated to $\pp_{r^{(l)}}$, there is a unique infinite connected component which is the image of an infinite quadrangulation with infinite simple boundary $\qq^{(l)}$. We define $\bb^{(l)}$ as the union of
	\begin{itemize}
		\item All the vertices of $\pp_{r^{(l)}}$ that are image of vertices belonging to $\qq^{(l)}$.
		\item All the vertices that were explored in this process that have image belonging to a face in $\qq^{(l)}$.
	\end{itemize}
	Note that $\bb^{(k)}$ is an interval in the boundary of $\qq^{(k)}$.
	We define $\pp^{(l)}\subseteq\qq_\infty$ as the union of: the complement of the image of $\qq^{(l)}$ and the image of $\bb^{(l)}$ (also known as filled-in\footnote{ \label{filled-in}Filled-in explorations here refers to the explorations with target such after each step of the exploration we reveal the parts that are unexplored that do not contain the target point.} of the explored part in the literature).
\end{itemize} 
See \Cref{peel} for an idea of the peeling.
\begin{figure}[h!]
	\includegraphics[scale=0.5]{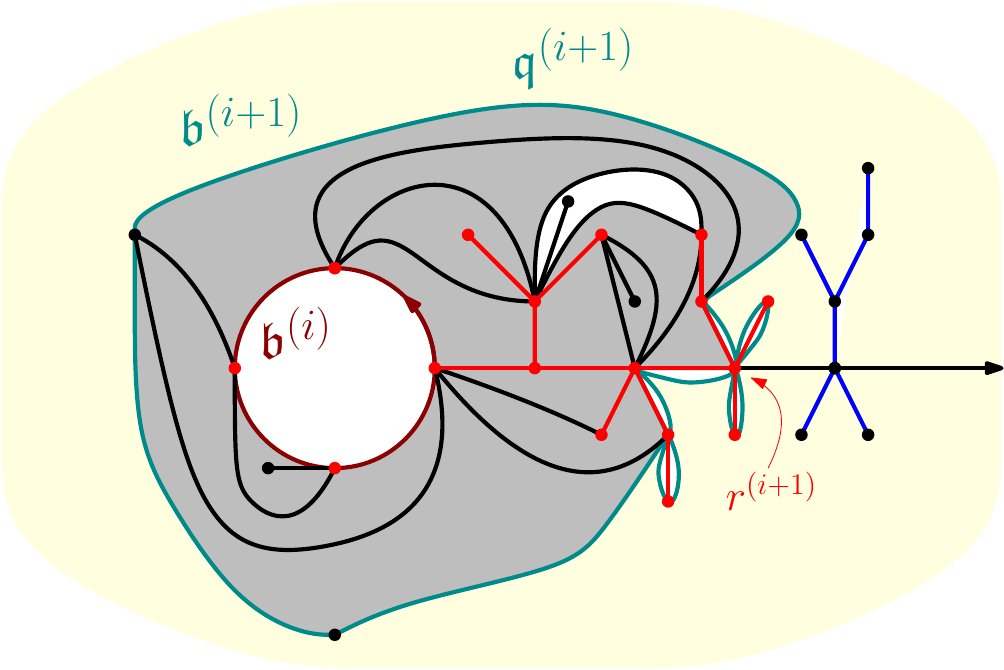}
	\caption{Sketch of the peeling procedure. The grey faces are the faces peeled at stage $i$. The set of red vertices form $\pp_{r^{(i+1)}}$ and the yellow region corresponds to $\qq^{(i+1)}$.}
	\label{peel}
\end{figure}

By construction of the peeling we obtain the following lemma.
\begin{lemma}\label{l.ball_ss_pelota}
	We have that for all $l\in \N$,  $\bigcup_{v \in \pp} B(v,l) \subseteq \pp^{(l)}$.
\end{lemma}
Define the random variable $\rqq^{(l)}$ with value $\qq^{(l)}$ rooted at the edge with image $(r^{(l)}r^{(l-1)})$ when $(\rqq,\rtt)= (\qq_{\infty,\infty}^b,\tt_\infty)$. 
Also define the random variable $\rtt^{(l)}$ equal to the unexplored part of $\ttm_\infty$ (an infinite tree) rooted at the edge $(r^{(l)}r^{(l+1)})$ when $(\rqq,\rtt)= (\qq_{\infty,\infty}^b,\tt_\infty)$.
This peeling is Markovian in the following sense.
\begin{lemma}
	$\rqq^{(l)}$ is distributed as the $\UIHPQ$ and $\rtt^{(l)}$ is distributed as the infinite critical geometric tree.
\end{lemma}
\begin{rem}
	This can be establish as a Markovian property on a decorated map with simple boundary (see \Cref{Fig:Markov}), however we skip this, since it will not be needed in the proof.
	\begin{figure}
		\includegraphics[scale=0.47]{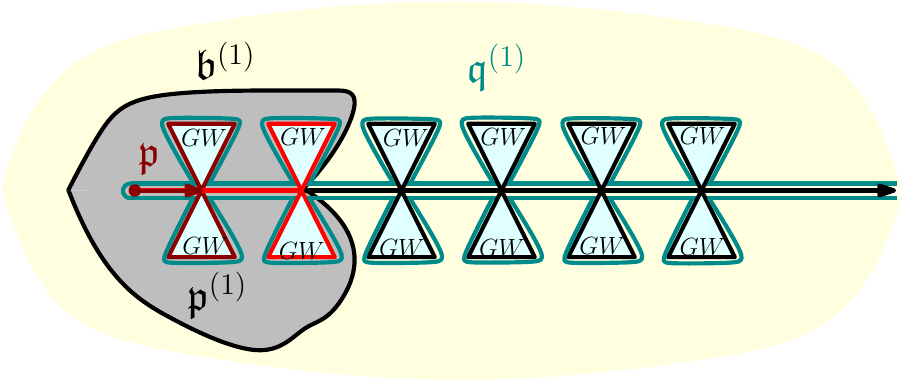}  \includegraphics[scale=0.47]{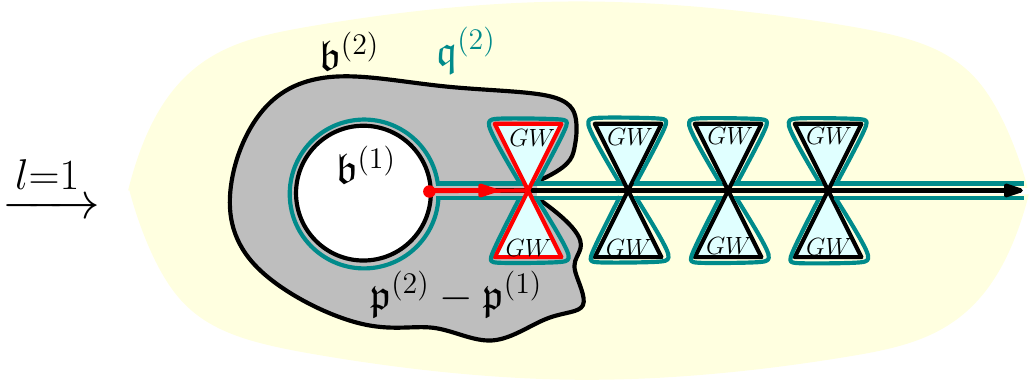}\\
		\includegraphics[scale=0.47]{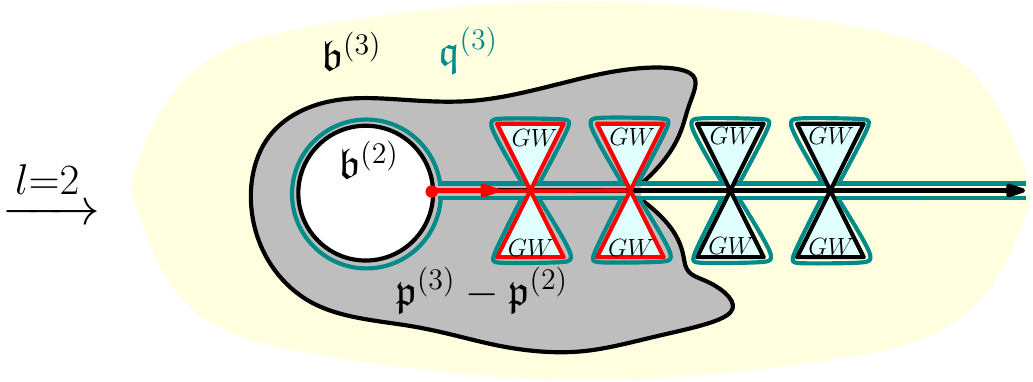}
		\caption{The Markovian property seen from the decorated map, where in each successive step we take out the $\pp^{(l)} -\pp^{(l-1)}$ discovered by the peeling in the preceding step. The red part represents the points that are covered up to $r^{(l)}$ in the tree, the grey parts represent the increment between balls of $\pp$ in each stage of the peeling, and the black boundary represents the intervals $b^{(l)}$. The map $\qq^{(l)}$ consists in the yellow part delimited by the black outer boundary and the green line that is not contained in the grey part.} 
		\label{Fig:Markov}
	\end{figure}
\end{rem}

\subsubsection{Overshoot estimate in the peeling with target}
Now, we  prove that distances after gluing are bigger than $n/\log(n)$ for points belonging to trees whose roots are bigger than $n$ in the natural parametrization of the spine in the decoration. Recall from Definition \ref{d.subtreee} that $\ttm_\infty(m)$ is the finite subtree of $\ttm_\infty$ consisting on the spine $\cro{0,m}$ and all the trees that are attached to this part of the spine.

\begin{prop}\label{p.max_jump}
	For any $\alpha>1$ and $\epsilon>0$, we have that with high probability as $n\to \infty$
	\[
	\frac{n}{\log(n)^\alpha} \leq diam(\ttm_\infty(n))\leq \epsilon n.
	\]
\end{prop}

The proposition follows from the study the law of $r^{(l)}-r^{(l-1)}$.
\begin{lemma}\label{l.law_overshoot}
	There exists a (deterministic) constant $C$ independent of $l$ and $r$ such that
	\begin{align}
		\P(r^{(l)}-r^{(l-1)}\geq a\mid \mathcal F_{l-1})\leq Ca^{-1} + o(a^{-3}) \quad \forall a\in \N\setminus \{0\},
	\end{align}
	where $\mathcal F_{l}$ is the filtration associated to the peeling process up to time $l$.
\end{lemma}
Let us first prove that Proposition \ref{p.max_jump} follows from \ref{l.law_overshoot}.
\begin{proof}[Proof of Proposition \ref{p.max_jump}]
	We start by noting that the inequality to the right comes directly from the definition of the metric in the infinite continuous volume, \Cref{cor2}.
	For the left inequality it is enough to show that for every $\alpha>1$ with high probability as $n\to \infty$
	\begin{align*}
		\dd_{\mathsf{M}}(0,x)\geq \frac{n}{\log(n)^\alpha}+1 \ \ \forall x \in \ttm_\infty \backslash \ttm_\infty(n).\end{align*}
	We show this by studying the process $r^{(l)}$.
	
	An implication of Lemma \ref{l.law_overshoot} is that one can couple $r^{(l)}-r^{(l-1)}$ with an i.i.d. sequence $s^{(l)}$ such that $r^{(l)}-r^{(l-1)}\leq s^{(l)}$ and furthermore there exist a constant $C$ such that
	\begin{align*}
		\P(s^{(l)}\geq s)\leq Cs^{-1} + o(s^{-3}) \quad \forall s\in \N\setminus \{0\}.
	\end{align*}
	Since $r^{(l)} \leq \sum_{j=1}^l s^{(j)}$ it is enough to study the last sum.
	
	From \cite[Theorem 3]{H81} and \Cref{l.law_overshoot}, we obtain that for some choices of $a_l$ and $b_l$ the random variable $a_l^{-1}\sum_{j=1}^l s^{(j)} - b_l$ converges in distribution to an asymmetric Cauchy random variable, whose distribution we denote by $F_C$. This convergence has uniform error of size $o(1/\ln(l))$. Now, we use \cite[Theorem 8.3.1]{BGTT89} to chose $a_l= l$ and $b_l = C\log(l)$ (this is the same $C$ as in \Cref{l.law_overshoot}). 
	For $C^- = C- \nu$, with $0<\nu<C$, this gives
	\[
	\P\left(  \frac{\sum_{j=1}^l s^{(j)} }{l}- C\log(l)>-\nu\log(l)\right) = F_C(-\nu\log(l)) + o\left(\frac{1}{\log(l)}\right)
	\]
	If we consider $l = n/\log(n)$ this gives as $n\to \infty$ that
	\begin{align}\label{eq1}
		\P\left( \sum_{j=1}^{n/\log(n)} s^{(j)} \geq C^-n + o(n) \right)\leq F_C\left(-\nu\log\left(\frac{n}{\log(n)}\right)\right) + o\left(\log\left(\frac{n}{\log(n)}\right)^{-1}\right)
	\end{align}
	Here we use the fact that if $f(y)=y\log(y)$, then $f^{-1}(x)=\exp(W(x))$, where $W$ is the $W$ Lambert function. We also used the fact that $W(x)\sim \log(x)- \log(\log(x)) + \kappa\frac{\log(\log(x))}{\log(x)}$ as $x\rightarrow \infty$.
	
	We conclude by noting that the right hand side of \eqref{eq1} goes to $0$ as $n\to \infty$ and 
	\begin{align*}
		B_{l}(\mathsf{M})\cap \ttm_\infty\subseteq \ttm_{\infty}(r^{(l)}).
	\end{align*}
\end{proof}

We know discuss the proof of Lemma \ref{l.law_overshoot}.
\begin{proof}[Proof of Lemma \ref{l.law_overshoot}]
	We start by noting that conditioning on $\mathcal F_{l-1}$, we peel all faces of $\qq^{(l-1)}$ (that is an $\UIHPQ$) that have a vertex in the interval of the boundary $\bb^{(l-1)}=[x^{-},x^{+}]$. Let us first study how far this peeling goes in the boundary of $\qq^{(l-1)}$. As usual we associate to the boundary of $\qq^{(l-1)}$ a copy of $\Z$.
	
	Let us first define $\tilde O^+=\tilde x^+-x^+$, where $\tilde x^+$ is the largest point in $\bb^{(l-1)}$ that is peeled in the step $l$. We do the analogue definition for $\tilde O^-= x^--\tilde x^-$. Let us note that $\tilde O^+,\tilde O^-$ are stochastically dominated by $O^s$, where $O^s$ is the overshoot defined in Section \ref{ss.UIHPQ}.
	
	We now define $r^{(l),+}$, resp. $r^{(l),-}$, as the biggest $n\in \N$ such that $\ttm_\infty$ has a tree attached to its top, resp. bottom, part at step $l$. In other words, $r^{(l)}=\max\{r^{(l),+},r^{(l),-} \}$. Let us now note that for any $a\in \N\backslash\{0\}$
	\begin{align}\label{eq.smae}
		\P\left (r^{(l),+}-r^{(l-1)}\geq a\mid \mathcal F_{l-1}\right )= \P\left (\tilde O^+ >  \sum_{i=1}^{a-1} \tau_i \mid \mathcal F_{l-1}\right )\leq \P\left( O^s \geq  \sum_{i=1}^a \tau_i\right ),
	\end{align}
	where $\tau_i$ are i.i.d random variables having the law of twice the vertices of a critical geometric GW tree (not condition to survive, part 2) in \Cref{p.description_tree}). Inequality \ref{eq.smae} also hold when changing $r^{(l),+}$ by $r^{(l),-}$ Let us recall that the law of $\tau_i$ is the same as the law of the first time a simple random walk hits level $-1$.
	
	Thus, to finish the proof of this lemma we need to prove the following claim.
	\begin{claim}\label{c.real_q}
		We have that
		\begin{align*}
			\P\left (O^s\geq \sum_{i=1}^a \tau_i\right )\leq Ca^{-1} + o(a^{-3}).
		\end{align*}
	\end{claim}
	Before proving the claim let us note that it implies the lemma because
	\begin{align*}
		\P(r^{(l)}-r^{(l-1)}\geq a \mid \mathcal F_{l-1})&\leq 	\P(r^{(l),-}-r^{(l-1)}\geq a \mid \mathcal F_{l-1}) + 	\P(r^{(l),
			+}-r^{(l-1)}\geq a \mid \mathcal F_{l-1})\\
		&\leq 2\P\left (O^s \geq \sum_{i=1}^a \tau_i\right ).
	\end{align*}
	
\end{proof}
As we are peeling all faces that have a vertex belonging to $\bb^{(l-1)}$, let us call $O_{-}^l$, resp. $O_+^l$ the furthest point of the boundary of $\qq^{(l)}$ that got discover to the left, resp. right.

We provide now the proof of the claim.
\begin{proof}[Proof of Claim \ref{c.real_q}]
	We start by constructing $\tau_i$ by considering a simple random walk $(X(k):k\in \N)$ started at $0$ and define $\tau_i=t_{-i}-t_{-(i-1)}$, where $t_i$ is the first time a random walk hits level $i$. We can now note that under this coupling the events $\{\sum_{i=1}^a \tau_i\leq j\}$ and $\{\inf_{k\in\cro{0,j}} X(k) < -a \}$ are equal. Thus, we have that 
	\begin{align*}
		\P\left (O^s\geq \sum_{i=1}^a \tau_i \right )&= \P\left (\inf_{k\in \cro{0,O^s}} X(k)< -a \right )\\
		&\leq C \sum_{l=1}^\infty \P\left(\inf_{k\in \cro{0,l}}X(k)< -a \right)l^{-3/2}.
	\end{align*} 
	We separate the last sum according to whether $l>a^2$ or $l\leq a^2$. The sum in the first case, i.e. for $l>a^2$, can be easily bounded.
	\begin{align*}
		\sum_{l=a^2+1}^\infty  \P\left (\inf_{j\in \cro{0,l}}X(k)< -a \right ) l^{-3/2} \leq \sum_{l=a^2+1}^\infty l^{-3/2}\leq Ca^{-1}.
	\end{align*}
	
	For the sum of the second case, i.e. for $l\leq a^2$, we use that for a simple random walk
	\begin{align*}
		\P\left (\inf_{k\in \cro{0,l}}X(k)< -a\right )&= \P\left (\sup_{k\in \cro{0,l}}X(k)> a\right ) =  2\P\left(X(l)> a \right).
	\end{align*}
	And that furthermore by the Hoeffding's inequality we have that
	\begin{align*}
		\P\left(X(l)\geq a \right)\leq \exp\left(-\frac{a^2}{2l}\right). 
	\end{align*}
	Thus, we can finally bound
	\begin{align*}
		\sum_{l=1}^{a^2}\P\left (\inf_{k\in \cro{0,l}}X(k)< -a \right ) l^{-3/2}&\leq 2 \sum_{l=1}^{a^2}\exp\left(-\frac{a^2}{2l}\right)l^{-3/2}\\
		&\leq 2\int_0^{a^2+1}\exp\left(-\frac{a^2}{2x}\right) x^{-3/2} dx + \sup_{x\geq 0} \left(\exp\left(-\frac{a^2}{2x}\right) x^{-3/2}\right)\\
		&\leq  \frac{2}{a}\int_0^{1+1/a^2}\exp\left(-\frac{1}{2u}\right)u^{-3/2}du+o(a^{-3}) \leq Ca^{-1}  + o(a^{-3})
	\end{align*}
	where in the second inequality we used the fact that the monotonicity of the function changes only once (at the point $x=a^2/3$)
\end{proof}

\begin{cor}
	The infinite discrete volume shocked map is not absolutely continuous with respect to the $\UIHPQ_S$ with simple boundary.
\end{cor}
\begin{proof}
	To show this it is enough to study the distances in the boundary. On one side we know from Prop 6.1 \cite{CC18} that distances from the root to the boundary point labeled $n$ on the boundary of the $\UIHPQ_S$ scale as $\sqrt{n}$ and from \Cref{p.max_jump} in the discrete volume shocked map we know that for any $\alpha>1$
	\begin{align}
		\dd_{\qq_{\infty,\infty}^{b}}(0,n)< \mathfrak{c}\sqrt{n}<\frac{n}{\log(n)^\alpha} < \dd_{\qq_{\infty}^{\textsf{T},\infty}}(0,n)
	\end{align}
	From this we conclude.
\end{proof}

\section{Finite volume for the discrete map}\label{s.f_volume_d_m}

The objective of this section is to transfer Proposition \ref{p.max_jump} to the case of finite volume. To prove this, we first note that the event that any two different points on the tree are at positive distance on the ITQ, only depends on the tree itself and on the behaviour of the quadrangulation with a simple boundary at close distance from the boundary.  Then we use the techniques of \cite{BCFS} to prove that close to the tree, the behaviour of both the tree and the quadrangulation with a simple boundary are not so different from their infinite counterpart. 

\subsection{"Typical case for finite volume uniform tree is not-unlikely for infinite volume uniform tree"} \label{ss.tree}In this subsection, we construct an exploration of a uniformly chosen tree and we show that the result of this exploration is not unlikely to be seen in the infinite critical geometric tree.

Take $\rtt$ a tree of size $k$ and $(C_k(i))_{i\in \cro{0,2k}}$ its contour function. We mark the vertex $\bar v$ as the vertex visited at time $k$ by the contour function. For any $m\in \N$, we define $\rtt^{(m)}$ as follows. If $d_\rtt(0,\bar v)\leq m$, $\rtt^{(m)}$ is equal to $\rtt$. However if $d_\rtt(0,\bar v)>m$, then we take $\bar \rtt^{(m)}$ the subtree defined as the connected component of $\overline{v}$ in  $\rtt \backslash \Ball{\rtt,0}{m}$  rooted at the unique edge where this connected component is attached to $\Ball{\rtt,0}{m}$. We define $\rtt^{(m)}$ as the (marked and rooted) tree generated by the vertices on $\rtt \backslash \bar \rtt^{(m)}$, marked on the corner to where $\bar \rtt^{(m)}$ is attached and rooted in the same (oriented) edge as the $\rtt$ was. This finally allows us to define for $j\leq k$ the tree $\rtt_{j}$ as the tree $\rtt^{(\tilde m)}$ where $\tilde m$ is the first $m$ such that $\rtt^{(m)}$ has size greater than or equal to $j$.

In fact, it is possible to compute the probability of $\rtt_{j}$. This is given in the following lemma.

\begin{lemma} For any $j,k\geq0 $, take $\tt$ a random uniform tree of size $k\in \N$. We have that
	\begin{align*}
		\P\left( \rtt_j=t_j \right)= \frac{\frac{1}{k-r_j+1}\binom{2k-2r_j}{k-r_j}}{\frac{1}{k+1}\binom{2k}{k} },
	\end{align*}
	where $r_j\geq j$ is the size of $t_j$, and $t_j$ is a marked rooted tree such that the size of $(t_j)^{(m-1)}$ is less than or equal to $j$, where the exploration goes to the base point of the marked corner.
\end{lemma}
\begin{proof}
	We analyse the probability of the Dyck path associated to $\rtt$, for this consider the Dyck path associated to $\rtt_j$ and notice that from the marked corner on it we can identify the place where to insert the Dyck path of unexplored tree with size $k-r_j$. To conclude we use that the number of plane trees with $n$ edges are counted by the $n$-th Catalan number.
\end{proof}

For the next lemma, we need to explore the infinite critical geometric tree $\rtt$. In this case, to define $\rtt_j$, we only need to define $\bar \rtt^{(m)}$, this is done by taking the (unique) infinite connected component of $\rtt \backslash \Ball{\rtt,0}{m}$. We can now prove that the finite volume exploration is not-unlikely for the infinite volume one.
\begin{lemma}\label{l.nuTree}
	Take $\P_{k,j}$ (resp. $\P_{\infty,j}$) the law of $\rtt_{j}$ where $\rtt$ is a tree with size $k$ (resp. an infinite critical geometric tree). Then for any $\epsilon,\delta>0$ and $k\in \N$, there exists a set of trees $T^k:=T^{k,\epsilon,\delta}$ and a deterministic constant $K_{\epsilon,\delta}^T$ such that
	\begin{align}\label{c.prob_tree}
		\P_{k,(1-\epsilon) k}\left(\rtt_{(1-\epsilon) k} \notin T^k \right) \leq \delta, 
	\end{align}
	and such that for any $t\in T^k$ we have that
	\begin{align}\label{c.RN_tree}
		\dfrac{\P_{k,(1-\epsilon) k}(t)}{\P_{\infty,(1-\epsilon)k}(t)}\leq K_{\epsilon,\delta}^T.
	\end{align}
\end{lemma}
\begin{proof}
	We define $T^k$ to be the set of trees with size bigger than or equal to $(1-\epsilon)k$ and smaller than or equal to $(1-\gamma)k$, where $\gamma:= \gamma(\delta)$ is a parameter to be tuned, meaning that $\rtt_{(1-\varepsilon)k}\in T^k$ leaves a macroscopic part of $\rtt$ to be explored.
	We prove that \eqref{c.prob_tree} is satisfied, by studying the continuous limit of the trees. From the definition of $\rtt_{(1-\epsilon)k}$, we see that it has size bigger than or equal to $(1-\epsilon)k$, so we just need to prove that,  with high probability, it has size smaller than or equal to $(1-\gamma)k$. Consider the contour function $C_k$ of the tree $\rtt_k$ into play, we know from Theorem 2.5 in \cite{LG05} that, for the topology of uniform convergence
	\begin{align}\label{Donsker}
		\widetilde{C}_k =\left( \frac{C_k(2kt)}{\sqrt{2k}} : t\in [0,1]\right) \xrightarrow[k\rightarrow\infty]{law} (\mathbbm{e}_t :  t  \in [0,1]),
	\end{align}
	where $(\mathbbm{e}_t :  t  \in [0,1])$ is a standard Brownian excursion.
	Let us now describe how to read $\rtt_j$, the filled-in exploration of the tree with target point $\overline{v}$  using $\widetilde{C}_k$, to do that define $j_C:= \frac{j}{2k}$. To construct, $\rtt^{(m)}$, the exploration of the ball of radius $m$ in the tree of size $k$, we expose the values of the function in the set $\widetilde{C}_k^{-1}([0,\sqrt{2k} m]) \subseteq [0,1]$ and then we expose the values of $\widetilde C$ in all connected components of $[0,1]\backslash\widetilde{C}_k^{-1}([0,\sqrt{2k} m])$ that do not contain the point $1/2$, we mark the corner where the unseen interval should be connected. Denote $\hat L(\widetilde{C},m)$, one minus the length of that unexplored interval, take $m_{j}(\widetilde C)$ the infimum over $m\in \R$ such that $\hat L(\widetilde{C},m) >j_C$, we define $L(\widetilde{C}, j_C)=\hat L(\widetilde{C},m_j(\widetilde{C}))\geq j_C$. Note that this construction can be done for any renormalised contour function $\widetilde{C}$ and any $j_C\in [0,1]$. Finally, we remark that for any renormalised contour function $\widetilde{C}$ the function $\hat L(\widetilde{C},\cdot)$ is increasing and c\`adl\`ag.
	
	Thanks to Skorohod's representation theorem, we may assume that we work on a probability space where the convergence \eqref{Donsker} holds almost surely. The fact that $\hat L(\widetilde{C},\cdot)$ is increasing and c\`adl\`ag implies that in this coupling for any $m\in \R$
	\begin{align*}
		\lim_{l\nearrow m}\hat L(\mathbbm{e},l) \leq \liminf_{ k \to \infty} \hat L(\widetilde{C}_k,m) \leq  \limsup_{k\to \infty} \hat L(\widetilde{C}_k,m) \leq \hat L(\mathbbm{e},m).
	\end{align*}
	This implies that $0<j_C<1$
	\begin{align}\label{ineq:length}
		L(\mathbbm{e},j_C) \leq \liminf_{ k \to \infty} L(\widetilde{C}_k,j_C) \leq  \limsup_{k\to \infty} L(\widetilde{C}_k,j_C) \leq \lim_{\ell\searrow j_C}L(\mathbbm{e},\ell).
	\end{align}
	
	Now, notice that given $\delta>0$ one can find $\gamma>0$ such that $\P(L(\mathbbm{e},1-\epsilon/2)<1-\gamma) \geq 1-\delta$. But thanks to \eqref{ineq:length} one finds that for $k$ big enough has that 
	$\P(L(\widetilde C_k,1-\epsilon)<1-\gamma) \geq 1-\delta$, which let us conclude that \eqref{c.prob_tree} holds when one takes
	\begin{align*}
		T_k:=\{\rtt: \rtt \text{ has } k \text{ vertices and } \rtt_{(1-\epsilon)k} \text{ has at most $(1-\gamma)k$ vertices}\}.
	\end{align*}.
	
	For the second part take $k'\in \N$, we use \Cref{l.nuTree} and note that if $t$ has positive probability for $\P_{k,(1-\epsilon)k}$, then it has positive probability for any $\P_{k+k', (1-\epsilon) k}$. Then we use the fact that the $n$-th Catalan number behaves asymptotically like $4^n n^{-3/2}\pi^{-1/2}$ as $n\to \infty$, and we obtain 
	\[
	\dfrac{\P_{k,(1-\epsilon) k}(t)}{\P_{k+k',(1-\epsilon)k}(t)} \approx \left(\frac{k}{k-|t|}\right)^{3/2} 4^{-|t|}\times \left(\frac{k+k'}{k+k'-|t|}\right)^{-3/2}4^{|t|}\leq  \gamma^{-3/2}.
	\]
	Since this bound is independent of $k'$, we obtain the result by taking $k'\to\infty$.
\end{proof}

\subsection{"Typical case for quadrangulation with a boundary is not unlikely for the $\UIHPQ$"}\label{ss.quad}

The idea now, is to take a big subset of the boundary on a quadrangulation with a simple boundary and show that a small ball centred around this subset in also not unlikely to happen on a $\UIHPQ$. An important caveat is that in order to use this result to study the ITQ, we need to obtain bounds that are uniform on the unexplored size of the boundary.

Now, let us take $a,b,f\in \N$, $\sigma>0$, 
\begin{align*}
	I_{\sigma,f}= I_{\sigma, f,a,b}:=\left [-\pi \frac{a}{\lfloor \sigma \sqrt{f} \rfloor},\pi \frac{b}{\lfloor \sigma \sqrt{f} \rfloor}\right ]
\end{align*}
and $(\qq_{f,\sigma \sqrt{f}},\gb)$ where $\qq_{f,\sigma \sqrt{f}}$ is a uniformly chosen quadrangulations with a simple boundary  of size $2 \lfloor \sigma \sqrt{f} \rfloor$ and $f$ internal faces and $\gb$ is the curve $\gb: \S^1 \to \partial \qq_{f,\sigma \sqrt{f}}$ that goes through the boundary at constant speed. We identify $I_{\sigma,f}$ with the portion of the boundary of $\gb(e^{i I_\sigma,f})$ and we denote by $r_0$ the radius of the smallest filled-in ball centred at the root vertex that contains $I_{\sigma,f}$. For $r>0$ define $\bar {I}^{r}_{\sigma,f}$ as the connected component\footnote{We set $\bar {I}^{r}_{\sigma,f}$ to be empty if the distance between $\gb(\pi)$ and $I_{\sigma,f}$ is less than or equal to $\epsilon$.} containing $\gb(e^{i\pi})$ in the complement of the ball centred at the root vertex and with radius $r_0+r$. We define $I_{\sigma, f}^{r}$ as the marked map $(\qq_{f,\sigma \sqrt{f}}\backslash \bar {I}^{r}_{\sigma,f},v_1, v_2 )$\footnote{Here we consider the suppression of elements of $\bar {I}^{r}_{\sigma,f}$ minus the boundary points; meaning that $I_{\sigma, f}^{r}\cap \bar {I}^{r}_{\sigma,f} = \partial\bar {I}^{r}_{\sigma,f}$}. Here, $v_j=\gb(e^{it_j})$, where $t_1$ (resp. $t_2$) are the bigger, (resp. smaller), $t$ such that $\gb(e^{it})\in  \bar {I}^{r}_{\sigma,f}$.

\begin{figure}[h!]
	\includegraphics[scale = 0.5]{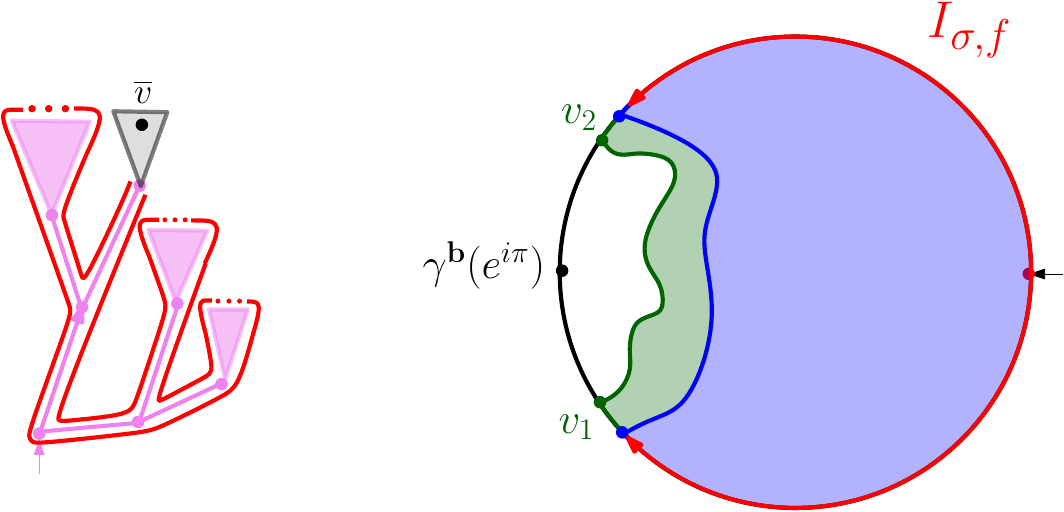}
	\caption{Explorations of \Cref{ss.tree} and \Cref{ss.quad}. \textbf{Left:} Filled-in exploration of the tree with target $\overline{v}$, where $\rtt_j$ is coloured in pink with its contour exploration in red. \textbf{Right:} Filled-in exploration of $I_{\sigma,f}$ with target point $\gb(e^{i\pi})$; the blue region represents the first filled-in ball covering $I_{\sigma,f}$ of radius $r_0$ and the green region is what is added to obtain the filled-in ball centred at the root vertex of radius $r_0+r$.}
\end{figure}

In this context, we need to be a little more careful to define $I_{\sigma,f}$ in the case when $f=\infty$, i.e., when taking a $\UIHPQ$ $(\qq_{\infty,\infty}, \gb)$ in that case 
\begin{align*}
	I_{\infty}:=I_{\infty,a,b}=\left[-a,b \right].
\end{align*}
In this case, we identify $I_{\infty}$ with the $\gb(I_{\infty})\subseteq\qq_{\infty, \infty}$. We define $I_{\infty}^r$ in an analogous way to the finite case.

For the following lemma we need to define two probability laws, $\Q_{f,\sigma, a,b,rf^{1/4}}$ and $\Q_{\infty,\infty,a,b,rf^{1/4}}$ as the law of $I_{\sigma,f}^{rf^{1/4}}$ and $I_{\infty}^{rf^{1/4}}$ respectively.	
\begin{lemma}\label{l.control_map_with_boundary}
	For every $\delta>0$ and $ 0<\eta < \sigma$ there exists $f_0\in \N$ and $r>0$ such that for all $f\geq f_0$ the following occurs.  Take $a,b\in \N$ with $a+b<(\sigma-\eta)\sqrt{f}$, there exists a set marked of maps $Q^{f}:=Q^{f,\eta,\delta, r}$ and a deterministic constant $K_{\delta,\eta}^b>0$ such that
	\begin{align}\label{c.prob}
		\Q_{f,\sigma,a,b,rf^{1/4}}\left (q\notin Q^f \right )\leq \delta,
	\end{align}
	and for any $q \in Q^f$
	\begin{align}\label{c.RN}
		\frac{\Q_{f,\sigma,a,b,rf^{1/4}}(q)}{\Q_{\infty,a,b,rf^{1/4}}(q)} \leq K_{\eta,\delta}^b
	\end{align}
\end{lemma}
\begin{proof}
	Define as $q_{f,\ell}$ as the cardinal of the set of quadrangulations with simple boundary of size $2\ell$ and $f$ internal faces. It is known \cite[Proof of Lemma 10.]{BCFS}
	\[
	q_{f,\ell}\sim \frac{\sqrt{3}}{2\pi}12^f\left(\frac{9}{2}\right)^\ell f^{-5/2}\ell^{1/2}\exp\left(-\frac{9\ell^2}{4f}\right)
	\]
	as $f$ and $\ell$ tend to infinity.
	
	We will study the filled-in exploration with target, meaning that we fill the connected components that do not contain the target point: the middle point of the perimeter. This is enough since the exploration without filling the holes is deterministic in the filled-in exploration.
	\begin{align*}
		\Q_{f,\sigma,a,b}(q) = \frac{q_{m,\ell}}{q_{f,\sigma\sqrt{f}}} 
	\end{align*}
	where $\ell:=\ell(q)= |\partial q_{in} |+ \lfloor\sigma\sqrt{f}\rfloor - |\partial q_{out}|$, where $\partial q_{out}$, resp. $\partial q_{in}$  is the segment of the boundary of $q$ written as $[v_1,v_2]\supseteq I$, resp. $(v_2,v_1)$ (i.e. that does not intersect $I$); and where $m:=m(q)$ is equal to $f$ minus the number of inner faces of $q$.
	
	We define \begin{align*}
		Q^f:=Q^{f,\eta,\delta, r} =  \left \{ q : \alpha \sqrt{f} \leq  \ell(q)\leq \frac{\sqrt{f}}{\alpha} , m(q)\geq \alpha f\right \},
	\end{align*}
	where $\alpha$ is a function of $\delta$, $\eta$ and $r$. In words, the event $Q^f$ is where the unexplored part has macroscopic size. 
	
	We need to prove properties \eqref{c.prob} and \eqref{c.RN}.  The fact that there exists $r>0$ such that property \eqref{c.prob} holds follows directly from Lemma 9 \cite{BCFS}, since our exploration is a fast-forward stage of the exploration used by them.
	
	We are left prove that property \eqref{c.RN} holds. Consider $q$ such that $q\in Q^f$, we claim that 
	\begin{align}\label{c.RN_cond}
		\frac{\Q_{f,\sigma,a,b,rf^{1/4}}(q)}{\Q_{\infty,a,b,rf^{1/4}}(q)} \leq K_{\epsilon,\delta}^b
	\end{align}
	Consider $f'\in \N$ (large enough) then by noting that if a marked map $q$ has positive probability for $\Q_{f,\sigma,a,b}$, then it also has positive probability for $\Q_{f+f',\sigma,a,b}$ we note that
	\begin{align}\label{e.RND_Q}
		&\frac{\Q_{f,\sigma,a,b,rf^{1/4}}(q)}{\Q_{f+f',\sigma,a,b,rf^{1/4}}(q)}\\
		\nonumber&= \frac{q_{f+f',\sigma\sqrt{f+f'}}}{q_{f,\sigma\sqrt{f}}}  \times \frac{q_{m,\ell}} {q_{m+f',\ell + \sigma(\sqrt{f+f'}-\sqrt{f})}}\\
		\nonumber&\left(\frac{f+f'}{f}\frac{m}{m+f'}\right)^{-5/2}\left(\frac{\sqrt{f+f'}}{\sqrt{f}}\frac{\ell}{\ell+\sigma\left(\sqrt{f+f'}-\sqrt{f}\right)}\right)^{1/2}\\
		\nonumber&\quad\times\exp\left(\frac{9}{4}\left(\frac{\left(\ell+\sigma\left(\sqrt{f+f'}-\sqrt{f}\right)\right)^2}{m+f'}-\frac{\ell^2}{m}\right)\right)\\
		\nonumber&\sim \left(\frac{f}{m}\right)^{5/2}\left(\frac{\ell}{\sigma\sqrt{f}}\right)^{1/2}\exp\left(\frac{9}{4}\left(\sigma^2-\frac{\ell^2}{m}\right)\right)\\
		\nonumber&\leq \left(\frac{f}{m}\right)^{5/2}\left(\frac{\ell}{\sigma\sqrt{f}}\right)^{1/2}\exp\left(\frac{9}{4}\sigma^2\right)\\
		\nonumber&\leq \alpha^{-3}\sigma^{-1/2}\exp\left(\frac{9}{4}\sigma^2\right)
	\end{align}

	Again this bound does not depend on $f'$, so taking the limit we obtain the result. 
	Here we used a ``diagonal'' version of the convergence to the $\UIHPQ$ with simple boundary when the area and the boundary tend to infinite simultaneously, with the boundary of order square root of the area . This ``diagonal'' version follows from Prop. 2.6 and Lemma 2.7 in \cite{GM}.
	
\end{proof}

\subsection{"Close to the tree a tree decorated quadrangulation is not unlikely for the ITQ"} Now, we use the results before to prove that high probability events that depend only on small neighbourhoods of a finite part of the tree in the ITQ also have high probability for a finite tree decorated quadrangulation.

\begin{prop}\label{p.TDM_small_diameter}
	Let $(\rqq_f, \rtq_\sigma)$ be a tree decorated quadrangulation where $\rqq_f$ has $f$ faces and the tree $\rtq_\sigma$ is of size $\sigma \sqrt f$, with $0<\sigma<\infty$. 
	For any $\alpha>1$ and $\epsilon>0$, we have that with high probability as $f\to \infty$
	\begin{align}\label{e.size_tree_finite}
		\frac{f^{1/4}}{(\log(f))^\alpha}\leq diam(\rtm_{\sigma})\leq \epsilon f^{1/4}.
	\end{align}
\end{prop}
\begin{proof}
	Let us first look at the diameter of the exploration tree $(\rtm_{\sigma})_{\frac{3}{4}\sigma \sqrt{f}}$, with the notation of Subsection \ref{ss.tree}; i.e. the first tree of size bigger than $\frac{3}{4}\sigma\sqrt{f}$ in the filled-in exploration of the tree $\rtm_{\sigma}$. We define the event $E(f)$ as
	\begin{align}\label{e.size_finite_tree}
		\frac{f^{1/4}}{(\log(f))^\alpha}\leq diam(\rtm_{\sigma})_{\frac{3}{4}\sigma                \sqrt{f}}\leq \frac{\epsilon}{2} f^{1/4}, 
	\end{align}
	We note that if with high probability $E(f)$ holds, then we can conclude, by triangular inequality and the rerooting invariance, that \eqref{e.size_tree_finite} also holds. We note that the event $E(f)$ only depends on an $\frac{\epsilon}{2} f^{1/4}$ neighbourhood of $(\rtm_{\sigma})_{\frac{3}{4}\sigma \sqrt{f}}$. Using the bijection, to an independent pair $(\qq^b_f, \rtt_{\sigma \sqrt f} )$, let us define $a$ and $b$ such that the contour function of $\rtt_{\sigma \sqrt f}$ visits a vertex of $(\rtt_{\sigma \sqrt f})_{\frac{3}{4}\sigma \sqrt f}$ at time $b$ but not at $b+1$, and visits a point of $(\rtt_{\sigma \sqrt f})_{\frac{3}{4}\sigma \sqrt f}$ at $\lfloor \sigma \sqrt{f}\rfloor- a$ but not at $\lfloor \sigma \sqrt{f}\rfloor- a-1$.
	
	All this gives that $E(f)$ depends only on $\left (I_{\sigma,f}^{\epsilon f^{1/4}},(\rtt_{\sigma \sqrt f})_{\frac{3}{4}\sigma \sqrt f}\right)$. 
	We define the event $E_{\infty}(f)$ as
	\begin{align*}
		\frac{f^{1/4}}{(\log(f))^\alpha}\leq diam((\rtm_{\infty})_{\frac{3}{4}\sigma \sqrt{f}})\leq \frac{\epsilon}{2} f^{1/4}.
	\end{align*}
	By Lemmas \ref{l.nuTree} and \ref{l.control_map_with_boundary} together with the definition of $T^k$, we see that the probability of the complement of $E(f)$ is upper bounded by
	\begin{align*}
		2\delta + K_{3/4,\delta}^T K^b_{\eta, \delta}\P((E_\infty(f))^c),
	\end{align*}
	where we first chose $\delta>0$, and then take $\eta<\delta$ such that $,(\rtt_{\sigma \sqrt f})_{\frac{3}{4}\sigma \sqrt f} \in T^{\sigma \sqrt f}$.

	The probability of the event $E_\infty(f)$ goes to $1$ as $f\to \infty$ thanks to  \Cref{p.max_jump} and the fact that with high probability, as $C\to \infty$, $(\rtm_{\infty})_{\frac{3}{4}\sigma \sqrt{f}}$ is contained in $\rtm_{\infty} (Cf^{1/4})$  and contains $\rtm_{\infty} (C^{-1}f^{1/4})$.
\end{proof}

\section{Finite continuous volume}\label{s.fc}
The objective of this section is to finally prove Theorem \ref{t.triviality-SM}. This result is proven in a way that is analogous to that of the finite discrete volume case, so we only give a quick idea of how to adapt the result. The key result that allows this is the following lemma.

\begin{lemma}\label{l.convergence_expectation}
	Take a sequence $(\P_n: n\in \N)$ and $(\Q_n: n \in \N)$  sequences of probability measures in a Polish space converging to $\P$ and $\Q$ respectively, where $\Q_n$ is absolutely continuous with respect to $\P_n$ ($\Q_n\ll \P_n$). Assume that $(X_n: n\in \N)$ is a sequence of random variables in the same probability space such that the law of $X_n$ is $\P_n$ and converge a.s. toward $X$ with law $\P$. Take $f$ and $g$ two continuous function such that\begin{align}\label{e.g_def}
		f(X_n)\frac{d \Q_n}{d\P_n}(X_n) \stackrel{L^1}{\to} g(X),	\end{align}
	where $g$ is a deterministic function. Then 
	\begin{align}\label{e.expected_value_g}
		\P\left[ g(X)\right] = \Q\left[f(X) \right].
	\end{align}
\end{lemma}
\begin{proof}
	We have that $\left|\E\left[ g(X) - \frac{d \Q}{d\P}(X)f(X)\right]\right |$ is upper bounded by
	\begin{align*}
		&\left|\E\left[ g(X) - \frac{d \Q_n}{d\P_n}(X_n)f(X_n) \right]\right| +  \left|\E\left[ \frac{d\Q_n}{d\P_n}(X_n)f(X_n) - \frac{d \Q}{d\P}(X)f(X)\right]\right|.\end{align*}
	We conclude by noting that the left  term goes to $0$ by hypothesis, and the right one goes to $0$ by the convergence of $\Q_n$ towards $\Q$.
\end{proof}
Which allows us to show
\begin{lemma}\label{l.convergence_RND}
	In the context of Lemma \ref{l.convergence_expectation}, assume that there is $\mathcal O$ an open set of the Polish space, such that for every continuous function $f$ with compact support in $\mathcal O$ there is a $g$ such that \eqref{e.g_def} holds, then for every such $f$ a.s.
	\begin{align*}
		g(X)= f(X)\frac{d\Q}{d\P}
	\end{align*}
\end{lemma}
\begin{proof}
	Note that \eqref{e.g_def} implies that $\frac{d \Q_n}{d\P_n}(X_n)$ converges a.s. to $g(X)/f(X)$ for any $X_n$ that converges to $X$ with  $f(X)\neq 0$. We conclude by using the characterisation of the integrals of the limit \eqref{e.expected_value_g}.
\end{proof}

We can finally give (an sketch of) the proof of Theorem \ref{t.triviality-SM}.
\begin{proof}[Proof of Theorem \ref{t.triviality-SM}]
	We follow as in Section \ref{s.f_volume_d_m}.
	\begin{itemize}
		\item \textbf{Typical case for the CRT is not unlikely for the infinite CRT} As before we explore a CRT with a marked point and we obtain a result analogous to that Lemma \ref{l.nuTree}, i.e. that this exploration is not unlikely to happen in the infinite CRT. This could be done directly in the continuous using the Brownian motion compared with a Brownian excursion. It also follows directly from Lemma \ref{l.convergence_RND}, as the Radon-Nikodym derivative is continuous with respect to the size of the leftover tree and the fact that one can make $T^k$ to be an open set.
		\item \textbf{Typical case for the Brownian disk is not unlikely for the Brownian half-plane} As before we mark a boundary point of the Brownian disk, explore it an obtain a result analogous to Lemma \ref{l.control_map_with_boundary}. To do that we use Lemma \ref{l.convergence_RND} for decorated metric spaces with the Gromov-Hausdorff-Uniform topology. As the explored sets are not themselves Brownian disks, we have to be careful. One just needs to restrict oneself to those maps that live in $Q^f$, in that case as the length of the boundary will remain bounded, one can have a subsequence where $\Q_{\infty,a,b,rf^{1/4}} \1_{Q_f}$ and $\Q_{,a,b,rf^{1/4}} \1_{Q^f}$ both converge. As the Radon-Nikodym derivative \ref{e.RND_Q} is continuous on the length of the map, we can apply \ref{l.convergence_RND} to conclude.
		\item \textbf{Typical case for the Shocked map near the boundary is not unlikely for the infinite Shocked map} For the final part, we note that the event where the diameter of the map is $0$ depends only on a small neighbourhood near the tree itself. And then we do as in the proof of Proposition \ref{p.TDM_small_diameter}, to see that we can do an exploration towards a mid-point of the tree and see that this exploration is not unlikely for the infinite shocked map and use Theorem \ref{t.main_section_3} to see that the diameter of the exploration of the tree is $0$. We then re-root our map and do the same exploration again, to conclude that the whole diameter of the tree is $0$.
	\end{itemize}
	To conclude the theorem notice that we just prove that after the glueing the CRT has diameter zero and since every path on the interior of the disk does not change its length, we can upper bound the distance of the glueing by the distance of the Brownian disk where the boundary is identified with one point.
\end{proof}

\appendix
\section{Markov property of the Brownian half-plane }
In this section, we discuss the Markov property of the Brownian half-plane on the special case where a filled-in ball is used. A general result of this type has already been announced in \cite{LR23}, however as the paper is not yet published we write a short proof here.
\begin{prop}\label{prop:markovBHP}
	Let $\HHH$ be a Brownian half-plane.  The filled-in ball with target point at infinity $B^{\bullet}_r(\HHH)$ and $\HHH'$ the complement with respect to $\HHH$, i.e. $\HHH' = \HHH\setminus B^\bullet_r(\HHH)$, are independent and moreover the $\HHH'$ properly rooted has the law of a BHP.
\end{prop}
\begin{proof}
	We already know that this is the case for the UIHPQ, $\qq$. If we renormalise the UIHPQ $n^{-1}\qq$, we have the convergence for Gromov-Hausdorff local topology to $\HHH$. As the pair $B^{\bullet}_r(n^{-1}\qq)$, and $\qq\backslash B^{\bullet}_r(n^{-1}\qq)$ are independent and converge in law to $B^{\bullet}_r(n^{-1}\qq)$ together with $\HHH\backslash B^{\bullet}_r(n^{-1}\qq)$ we conclude.
\end{proof}

\section{Convergence of the map with a simple boundary}\label{a.b}

In our proofs we make reference of Theorem 1 in \cite{BCFS}, we present it for completeness. Consider $\tilde{Q}_{f,p}$ a random uniform quadrangulation with $f$ internal faces and with a simple boundary of length $2p$.

\begin{thm}[Theorem 1 \cite{BCFS}]\label{t.convergence BD}
	For a sequence $(p_f:n\in \N)$ satifying $p_f\sim 2\alpha\sqrt{2f}$, it holds
	\[
	\left( \frac{9}{8f} \right)^{1/4}\tilde{Q}_{f,p_f} \xrightarrow[\;\;f\rightarrow\infty\;\;]{(d)} \mathfrak{D}_{3\alpha}
	\]
	in distribution for the Gromov-Hausdorff topology.
\end{thm}

In fact, here we use a strengthened version of this theorem where we put into play the Gromov-Hausdorff-Prohorov-Uniform distance (see \cite[Eq. (1.3)]{GM}) instead of the Gromov-Hausdorff distance.

The Gromov-Hausdorff-Prohorov-Uniform topology keeps track of both the area and perimeter measures of the map. The reason why this generalization of Theorem 1 \cite{BCFS} holds is the following. We studied the $\varepsilon$-restrictions ($\mathcal{R}_f^\varepsilon$)\footnote{Defined in Section 2.2. of \cite{BCFS}.} which are roughly the filled-in explorations started from an interior vertex with target point placed at 1/3 of the counter-clockwise perimeter and stopped the first time the exploration hits a point in between $1/3-\varepsilon$ and $1/3$ of the perimeter (see \cref{fig:1}). In order to control the perimeter and area of the complement of the $\varepsilon$-restriction ($\overline{\mathcal{R}}_f^\varepsilon$) we proved that with high probability they have the same order as the perimeter and area of the map, respectively, and they both go to zero as $\varepsilon$ goes to zero\footnote{This is a consequence of the volume and perimeter estimates given in Section 4.2 \cite{BCFS}} .

\begin{figure}[h!]
	\includegraphics[scale = 0.5]{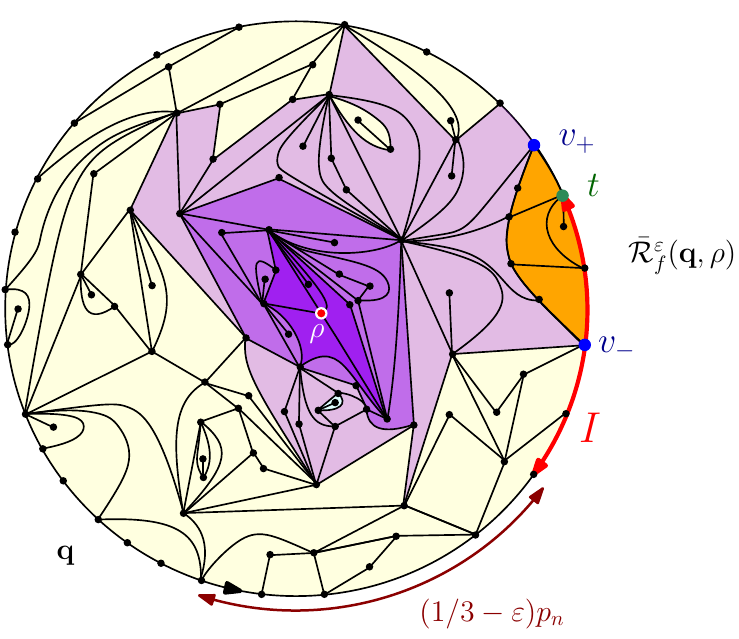}
	\caption{The exploration starts at $\rho$ and targets the green point $t$, it grows from lighter to darker purple until it hits for the first time the red segment at the point $v_-$, where the exploration stops. The complement of the restriction is highlighted in orange. The counter-clockwise perimeter starting from $t$ and going to $v_+$}
	\label{fig:1}
\end{figure}

\bibliographystyle{alpha} 
\bibliography{biblio}
\appendix

\end{document}